\documentclass[11pt,a4paper,oldfontcommands]{article}
\usepackage[utf8]{inputenc}
\usepackage[T1]{fontenc}
\usepackage{microtype}
\usepackage[dvips]{graphicx}
\usepackage{xcolor}
\usepackage{times}

\usepackage[english]{babel}
\usepackage{amsmath,mathtools}
\usepackage{amsfonts}
\usepackage{mathrsfs}
\usepackage{amsthm}
\usepackage{subfigure}
\usepackage{envmath}
\usepackage{amssymb}
\usepackage{dsfont}
\usepackage{bigints}
\usepackage{textgreek}
\usepackage{authblk}
\usepackage{titlesec}
\usepackage{enumitem}
\usepackage{url}

\usepackage[
breaklinks=true,
colorlinks=false,
bookmarks=true,bookmarksopenlevel=2]{hyperref}

\usepackage{geometry}
\newtheorem{theorem}{Theorem}[section]
\newtheorem{prop}[theorem]{Proposition}
\newtheorem{lem}[theorem]{Lemma}
\newtheorem{cor}[theorem]{Corollary}

\theoremstyle{remark}
\newtheorem{rem}[theorem]{Remark}

\numberwithin{equation}{section}

\renewcommand{\Re}{\operatorname{Re}}

\renewcommand{\Im}{\operatorname{Im}}

\newcommand*\diff{\mathop{}\!\mathrm{d}}

\DeclarePairedDelimiter\abs{\lvert}{\rvert}%
\DeclarePairedDelimiter\norm{\lVert}{\rVert}%

\makeatletter
\let\oldabs\abs
\def\abs{\@ifstar{\oldabs}{\oldabs*}}
\let\oldnorm\norm
\def\norm{\@ifstar{\oldnorm}{\oldnorm*}}
\makeatother

\newcommand*{\myemail}[1]{%
    \normalsize\href{mailto:#1}{#1}\par
    }

\allowdisplaybreaks

\titleformat{\section}[block]{\centering \scshape \large}{\thesection.}{0.3\baselineskip}{}
\titlespacing{\section}{0pt}{*5}{*2}

\titleformat{\subsection}[block]{\bfseries}{\thesubsection.}{.5em}{}
\titlespacing{\subsection}{0pt}{*2.5}{*1}

\titleformat{\subsubsection}[runin]{\itshape}{\normalfont \thesubsubsection.}{.5em}{}[.]
\titlespacing{\subsubsection}{0pt}{*2.5}{0.5em}

\titleformat{\section}[block]{\centering \scshape \large}{\thesection.}{0.3\baselineskip}{}
\titlespacing{\section}{0pt}{*5}{*2}

\titleformat{\subsection}[block]{\bfseries}{\thesubsection.}{.5em}{}
\titlespacing{\subsection}{0pt}{*2.5}{*1}

\titleformat{\subsubsection}[runin]{\itshape}{\normalfont \thesubsubsection.}{.5em}{}[.]
\titlespacing{\subsubsection}{0pt}{*2.5}{0.5em}

\title{The focusing logarithmic Schrödinger equation: analysis of breathers and nonlinear superposition}
\date{\vspace{-1cm}}

\author[]{Guillaume Ferriere}

\affil[]{IMAG, Univ Montpellier, CNRS, Montpellier, France \\ \myemail{guillaume.ferriere@umontpellier.fr}}

\begin{document}

\maketitle

\begin{abstract}
    We consider the logarithmic Schrödinger equation in the focusing regime. For this equation, Gaussian initial data remains Gaussian. In particular, the Gausson - a time-independent Gaussian function - is an orbitally stable solution. In the general case in dimension $d=1$, the solution with Gaussian initial data is periodic, and we compute some approximations of the period in the case of small and large oscillations, showing that the period can be as large as wanted for the latter. The main result of this article is a principle of nonlinear superposition: starting from an initial data made of the sum of several standing Gaussian functions far from each other, the solution remains close (in $L^2$) to the sum of the corresponding Gaussian solutions for a long time, in square of the distance between the Gaussian functions.
\end{abstract}

\section{Introduction}

\subsection{Setting}

We are interested in the \textit{Logarithmic Non-Linear Schrödinger Equation}
\begin{equation}
    i \, \partial_t u + \frac{1}{2} \Delta u + \lambda u \ln{\lvert u \rvert^2} = 0, \qquad \qquad
    {{u}_|}_{t=0} = u_{\textnormal{in}}, \label{foc_log_nls}
\end{equation}
with $x \in \mathbb{R}^d$, $d \geq 1$, $\lambda \in \mathbb{R} \setminus \{ 0 \}$. It was introduced as a model of nonlinear wave mechanics and in nonlinear optics (\cite{nonlin_wave_mec}, see also \cite{inco_white_light_log, log_nls_nuclear_physics, quantal_damped_motion, solitons_log_med, log_nls_magma_transp}). The case $\lambda < 0$ (whose study goes back to \cite{cazenave-haraux, Guerrero_Lopez_Nieto_H1_solv_lognls}) was recently studied by R. Carles and I. Gallagher who made explicit an unusually faster dispersion with a universal behaviour of the modulus of the solution (see \cite{carlesgallagher}). The knowledge of this behaviour was very recently improved with a convergence rate but also extended through the semiclassical limit in \cite{Ferriere__Wass_semiclass_defoc_NLS}. On the other hand, the case $\lambda > 0$ seems to be the more interesting from a physical point of view and has been studied formally and rigorously a lot (see for instance \cite{D'av_Mont_Squa_lognls, quantal_damped_motion}). In particular, the existence and uniqueness of solutions to the Cauchy problem have been solved in \cite{cazenave-haraux}.
Moreover, it has been proved to be the focusing case and, in this context, a usual question is the existence of stationary states called \emph{solitons} and their stability.
For this equation, we know that the so called \emph{Gausson}
\begin{equation}
    G^d (x) \coloneqq \exp \Bigl( \frac{d}{2} - \lambda \abs{x}^2 \Bigr), \qquad x \in \mathbb{R}^d, \label{expr_gausson}
\end{equation}
and its by-products through the invariants of the equation (translation in space, Galilean invariance, multiplication by a complex constant of modulus 1) are explicit solutions to \eqref{foc_log_nls} and bound states for the energy functional. Several results address the orbital stability of the Gausson as well as the existence of other stationary solutions to \eqref{foc_log_nls}; see e.g. \cite{nonlin_wave_mec, Cazenave_log_nls, D'av_Mont_Squa_lognls, Ardila__Orbital_stability_Gausson}.
However, these solutions are not the only explicit solutions: indeed, any Gaussian initial data remains Gaussian and can be made explicit up to a matrix ODE.
In this article, we study a particular class of time-dependent Gaussian solutions to \eqref{foc_log_nls} already introduced in \cite{nonlin_wave_mec} which are (almost) periodic in time in dimension $1$, so called \emph{breathers}, and we give a partial result for the nonlinear superposition of Gaussian solutions.

\begin{rem}[Effect of scaling factors] \label{rem_scaling}
    As noticed in \cite{carlesgallagher}, unlike what happens in the case of an homogeneous nonlinearity (classically of the form $\abs{u}^p u$), replacing $u$ with $\kappa u$ ($\kappa > 0$) in \eqref{foc_log_nls} has only little effect, since we have
    \begin{equation*}
        i \, \partial_t (\kappa u) + \frac{1}{2} \Delta (\kappa u) + \lambda (\kappa u) \ln{\lvert \kappa u \rvert^2} - 2 \lambda (\ln{\kappa}) \kappa u = 0.
    \end{equation*}
    The scaling factor thus corresponds to a purely time-dependent gauge transform:
    \begin{equation*}
        \kappa u(t,x) e^{-2it \lambda \ln{\kappa}}
    \end{equation*}
    solves \eqref{foc_log_nls} (with initial datum $\kappa u_0$). In particular, the $L^2$-norm of the initial datum does not influence the dynamics of the solution.
\end{rem}

\subsection{The Logarithmic Non-Linear Schrödinger Equation}

The Logarithmic \\ Schrödinger Equation was introduced by I. Bia{\l}ynicki-Birula and J. Mycielski \cite{nonlin_wave_mec} who proved that it is the only nonlinear theory in which \emph{the separability of noninteracting systems} hold: for noninteracting subsystems, no correlation is introduced by the nonlinear term. This means that for any initial data of the form $u_\textnormal{in} = u_\textnormal{in}^1 \otimes u_\textnormal{in}^2$, i.e. $u_\textnormal{in} (x) = u_\textnormal{in}^1 (x_1) \, u_\textnormal{in}^2 (x_2)$ for all $x_1 \in \mathbb{R}^{d_1}$, $x_2 \in \mathbb{R}^{d_2}$ and $x = (x_1, x_2)$, the solution $u$ to \eqref{foc_log_nls} (in dimension $d = d_1 + d_2$) is $u (t) = u^1 (t) \otimes u^2 (t)$ where $u^1$ and $u^2$ are the solutions to \eqref{foc_log_nls} in dimension $d_1$ (resp. $d_2$) with initial data $u_\textnormal{in}^1$ and $u_\textnormal{in}^2$ respectively.

They also emphasized that the case $\lambda > 0$ is probably the most physically relevant, and we will mathematically study this case in the rest of this paper.
For this case, the Cauchy problem has already been studied in \cite{cazenave-haraux} (see also \cite{Cazenave_semlin_lognls}). We define the energy space
\begin{equation*}
    W (\mathbb{R}^d) \coloneqq \{ v \in H^1 (\mathbb{R}^d), \abs{v}^2 \ln{\abs{v}^2} \in L^1 (\mathbb{R}^d) \}.
\end{equation*}
It is a reflexive Banach space when endowed with a Luxembourg type norm (see \cite{Cazenave_log_nls}).
We can also define the mass and energy for all $v \in W (\mathbb{R}^d)$:
\begin{gather*}
    M (v) \coloneqq \norm{v}_{L^2}^2, \qquad E(v) \coloneqq \frac{1}{2} \norm{\nabla v}_{L^2}^2 - \lambda \int_{\mathbb{R}^d} \abs{v}^2 (\ln{\abs{v}^2} - 1) \diff x.
\end{gather*}

\begin{theorem}[{\cite[Théorème~2.1]{cazenave-haraux}}, see also {\cite[Theorem~9.3.4]{Cazenave_semlin_lognls}}]
    For any initial data $u_\textnormal{in} \in W (\mathbb{R}^d)$, there exists a unique, global solution $u \in \mathcal{C}_b (\mathbb{R}, W(\mathbb{R}^d))$. Moreover the mass $M(u(t))$ and the energy $E(u(t))$ are independent of time.
\end{theorem}

Note that whichever the sign of $\lambda$ is, the energy $E$ has no definite sign. The distinction between focusing or defocusing nonlinearity is thus a priori ambiguous.
However, this ambiguity has already been removed by \cite{Cazenave_log_nls} (case $\lambda > 0$) and \cite{carlesgallagher} (case $\lambda < 0$). Indeed, in the latter, the authors show that all the solutions disperse in an unusually faster way with a universal dynamic: after rescaling, the modulus of the solution converges to a universal Gaussian profile.
On the other hand, it has been proved that there is no dispersion for large times for $\lambda > 0$ thanks to the following result, and hence it is the focusing case.
\begin{lem}[{\cite[Lemma~3.3]{Cazenave_log_nls}}]
    Let $\lambda > 0$.
    For any $k < \infty$ such that
    \begin{equation*}
        L_k \coloneqq \{ v \in W (\mathbb{R}^d), \norm{v}_{L^2} = 1, E(v) \leq k \} \neq \emptyset,
    \end{equation*}
    there holds
    \begin{equation*}
        \inf_{\substack{v \in L_k \\ 1 \leq p \leq \infty}} \norm{v}_{L^p} > 0.
    \end{equation*}
\end{lem}
This lemma, along with the conservation of the energy and the invariance through scaling factors (with Remark \ref{rem_scaling}), indicates that the solution to \eqref{foc_log_nls} is not dispersive, no matter how small the initial data are. For instance, its $L^\infty$ norm is bounded from below: to be more precise, there holds for all $t \in \mathbb{R}$ (see the proof of the above result)
\begin{equation*}
    \norm{u (t)}_{L^\infty} \geq \exp{\Bigl[ - \frac{E (u(t))}{2 M (u(t))} \Bigr]} = \exp{\Bigl[ - \frac{E (u_\textnormal{in})}{2 M (u_\textnormal{in})} \Bigr]}.
\end{equation*}

Actually, a specific Gaussian function \eqref{expr_gausson} called \emph{Gausson} and its by-products through the invariants of the equation and the scaling effect, defined by
\begin{equation*}
    G^d_{\omega, x_0, v, \theta} (t,x) = \exp \left[ i \left( \theta + 2 \lambda \omega t - v \cdot x + \frac{\abs{v}^2}{2} t \right) + \omega - \lambda \abs{x - x_0 - vt}^2 \right],
\end{equation*}
where $t \in \mathbb{R}, \, x \in \mathbb{R}^d$, and for any $\omega, \theta \in \mathbb{R}$, $x_0, v \in \mathbb{R}^d$, are known to be solutions to \eqref{foc_log_nls}, as proved in \cite{D'av_Mont_Squa_lognls} (and already noticed in \cite{nonlin_wave_mec}). It has also been proved that other radial stationary solutions to \eqref{foc_log_nls} exist in dimension $d \geq 3$ (see \cite{D'av_Mont_Squa_lognls}), but the Gausson is clearly special since it is the unique positive $\mathcal{C}^2$ stationary solution to \eqref{foc_log_nls} (also proved in \cite{D'av_Mont_Squa_lognls}) and also since it is orbitally stable (\cite{Ardila__Orbital_stability_Gausson}, following the work of \cite{Cazenave_log_nls}).

\begin{theorem}[{\cite[Theorem~1.5]{Ardila__Orbital_stability_Gausson}}]
    Let $\omega \in \mathbb{R}$.
    For any $\varepsilon > 0$, there exists $\eta > 0$ such that for all $u_0 \in W (\mathbb{R}^d)$ satisfying
    \begin{equation*}
        \inf_{\theta, x_0} \norm{u_0 - e^{\omega + i \theta} G^d (. - x_0)}_{W ( \mathbb{R}^d )} < \eta,
    \end{equation*}
    the solution $u(t)$ of \eqref{foc_log_nls} with initial data $u_0$ satisfies
    \begin{equation*}
        \sup_t \inf_{\theta, x_0} \norm{u (t) - e^{\omega + i \theta} G^d (. - x_0)}_{W ( \mathbb{R}^d )} < \varepsilon.
    \end{equation*}
\end{theorem}

\subsection{Main results}

From now on, we assume $\lambda > 0$.

\subsubsection{Existence of breathers}

The Gausson is an explicit important solution to \eqref{foc_log_nls}.
However, it is not the only solution that can be made explicit:

\begin{prop} \label{prop_expression_general_gaussian}
Any Gaussian initial data
\begin{equation}
    \exp \Bigl[ \frac{d}{2} - x^\top A_\textnormal{in} x \Bigr], \label{gauss_in_data}
\end{equation}
with $A_\textnormal{in} \in S_d (\mathbb{C}) \coloneqq \{ M \in M_d (\mathbb{C}), M^\top = M \}$ (where $^\top$ designates the transposition) such that $\Re A_\textnormal{in}$ is positive definite, gives rise to a Gaussian solution to \eqref{foc_log_nls} of the form
\begin{equation}
    u^{A_\textnormal{in}} (t,x) \coloneqq b(t) \exp \Bigl[ \frac{d}{2} - x^\top A (t) x \Bigr]. \label{gen_gauss_sol}
\end{equation}
\end{prop}

Furthermore, $b(t)$ is explicitly given by the knowledge of $A$ whereas the evolution of $A$ is given by a first-order matrix ordinary differential equations (see the system of equations (6.14), (6.15) and (6.16) of \cite{nonlin_wave_mec}).
In dimension $d=1$, this system simply becomes a system of two first-order ODEs. Yet, this system can be even more simplified, as it can be summarized into a particular second-order ODE, whose evolution can be better understood (see \cite{carlesgallagher, Bao_Carles_al__error_est}). In particular, as already noticed, an important feature is that both solutions to the system of two first-order ODEs and the second-order ODE are periodic, whatever the initial data are.

Before introducing this ODE, we define

\begin{equation*}
    \mathbb{C}^+ \coloneqq \{ z \in \mathbb{C}, \Re z > 0 \}.
\end{equation*}

\begin{prop} \label{prop_ODE_r}
    For any $\alpha \in \mathbb{C}^+$, consider the ordinary differential equation
    \begin{equation*}
        \Ddot{r}_\alpha = \frac{1}{r_\alpha^3} - \frac{2 \lambda}{r_\alpha}, \qquad r_\alpha (0) = \Re \alpha \eqqcolon \alpha_r, \qquad \dot{r}_\alpha (0) = \Im \alpha \eqqcolon \alpha_i.
    \end{equation*}
    It has a unique solution $r_{\alpha} \in \mathcal{C}^\infty (\mathbb{R})$ with values in $(0, \infty)$. Moreover, it is periodic.
\end{prop}

We can now properly define the breathers in dimension $d=1$.

\begin{prop}[Breathers for logNLS in dimension $1$] \label{prop_existence_breathers_d=1}
    For any $\alpha \in \mathbb{C}^+$, set
    \begin{equation*}
        u^{\alpha} (t,x) \coloneqq \sqrt{\frac{\alpha_r}{r_\alpha (t)}} \exp \Bigl[ \frac{1}{2} - i \Phi^\alpha (t) - \frac{x^2}{2 r_\alpha (t)^2} + i \frac{\dot{r}_\alpha (t)}{r_\alpha (t)} \frac{x^2}{2} \Bigr], \qquad t,x \in \mathbb{R},
    \end{equation*}
    where
    \begin{equation*}
        \Phi^\alpha (t) \coloneqq \frac{1}{2} \int_0^t \frac{1}{r_\alpha (s)^2} \diff s + \lambda \int_0^t \ln \frac{r_\alpha (s)}{\alpha_0} \diff s - \lambda t.
    \end{equation*}
    Then $u^{\alpha}$ is solution to \eqref{foc_log_nls} in dimension $d = 1$.
\end{prop}

We emphasize that these solutions to \eqref{foc_log_nls} are periodic in time up to a time-depending complex number of modulus 1; to be more precise, $u^\alpha \exp{\Bigl[ - i \Phi^\alpha (t) \Bigr]}$ (with $\Phi^\alpha (t)$ real) is periodic. Therefore, excluding the case $\alpha = (2 \lambda)^{- \frac{1}{2}}$ which is the Gausson, this explains why we call those solutions \emph{breathers}. It is worth pointing out that other solutions to \eqref{foc_log_nls} in dimension $1$, to be more precise
\begin{equation*}
    G_{\omega, x_0} (t,x) \coloneqq \exp{\Bigl[ - 2 i t \lambda \omega + \omega - \lambda (x-x_0)^2 \Bigr]}, \qquad t,x \in \mathbb{R},
\end{equation*}
for any $\omega, x_0 \in \mathbb{R}$, are known to be periodic in time, but their shape does not evolve ($\abs{G_{\omega, x_0}}$ is independent of time) contrary to these breathers.

From these results, the behaviour of these breathers and their shapes are in complete correlation with the knowledge of $r_\alpha$ and its evolution in time.
Therefore, it is interesting to study $r_\alpha$ more deeply. In particular, we can study its period, more precisely its evolution with respect to the initial data of $r_\alpha$. Indeed, it cannot be easily computed (its expression is rather complex, with an integral which cannot be explicitly computed), and one can wonder what regularity with respect to the parameters it has. Still, some approximations are available in two cases: the case of small oscillations (around $(2 \lambda)^{-\frac{1}{2}}$) and the case of "big" oscillations. An interesting feature in the latter was found as it appears that the period can become very large. We also postpone a discussion about the behaviour of $u^\alpha$ in this case in Section \ref{subsec_discussion_sol}.

\begin{theorem} \label{th_period_r}
    The period $T_\alpha$ of $r_\alpha$ is continuous with respect to $\alpha$. In addition, it satisfies $T_\alpha \longrightarrow \frac{\pi}{\sqrt{2} \, \lambda}$ when $\alpha \rightarrow (2 \lambda)^{-\frac{1}{2}}$ and $T_\alpha \sim \sqrt{\frac{\pi}{\lambda}} \, \alpha_r \exp{\Bigl[ \frac{\alpha_i^2}{4 \lambda} + \frac{1}{4 \lambda \alpha_r^2}} \Bigr]$ when $\alpha \rightarrow \infty$ or $\alpha_r \rightarrow 0$.
\end{theorem}

Those breathers can also be generalized in higher dimension $d \geq 2$.
Indeed, as already noticed in \cite{nonlin_wave_mec}, if the Gaussian initial data \eqref{gauss_in_data} is such that $\Re A_\textnormal{in}$ and $\Im A_\textnormal{in}$ commute, then $\Re A (t)$ and $\Im A (t)$ commute for any $t \in \mathbb{R}$ (since $[ \Re A (t), \Im A (t) ]$ is constant) and can be orthogonally co-diagonalized by the same time-independent orthogonal basis.
To this aim, define the orthogonal group $\mathcal{O}_d (\mathbb{R})$ and set
\begin{gather*}
    S_d (\mathbb{C})^{\Re +} \coloneqq \{ A \in S_d (\mathbb{C}); \Re A \textnormal{ is positive definite } \}, \\
    S_d (\mathbb{C})^{\Re ++} \coloneqq \{ A \in S_d (\mathbb{C})^{\Re +}; \Re A \textnormal{ and } \Im A \textnormal{ commute} \}.
\end{gather*}

\begin{prop}[Characterization of breathers for logNLS] \label{prop_existence_breathers}
    Let $d \in \mathbb{N}^*$ and \\ $A \in S_d (\mathbb{C})^{\Re ++}$.
    Set
    \begin{equation*}
        u^A_\textnormal{in} (x) = \exp \Bigl[ \frac{d}{2} - x^\top A x \Bigr],
        \qquad x \in \mathbb{R}^d,
    \end{equation*}
    and $u^A$ the solution to \eqref{foc_log_nls} with initial data $u^A_\textnormal{in}$.

    Then there exists $R \in \mathcal{O}_d (\mathbb{R})$ and $\alpha_1, \dots, \alpha_d \in \mathbb{C}^+$ such that for all $x \in \mathbb{R}^d$,
    \begin{equation*}
        u^A (t, Rx) = u^{\alpha_1} (t, x_1) \, u^{\alpha_2} (t, x_2) \dots u^{\alpha_d} (t, x_d).
    \end{equation*}
\end{prop}

However, in dimension $d \geq 2$, those generalized breathers are only a particular case among the more generalized class of Gaussian solutions of the form \eqref{gen_gauss_sol}. It has also already been pointed out (in \cite{nonlin_wave_mec}) that the positive definiteness of $\Re A (t)$ is preserved and the oscillations are bounded in amplitude. Even more, the spectrum of $\Re A (t)$ can be bounded by below by a positive time-independent constant $a_\textnormal{inf}$ but also by above, and along with the relation between $\Re A (t)$ and $\abs{b (t)}$,
\begin{equation*}
    \abs{b (t)} = \biggl( \frac{\det{\Re A (0)}}{\det{\Re A (t)}} \biggr)^\frac{1}{4},
\end{equation*}
we see that some features of the breathers remain true even though the periodicity is lost.

We denote those Gaussian solutions to \eqref{foc_log_nls} by taking care of the invariants and the scaling factor: for any $A_\textnormal{in} \in S_d (\mathbb{C})^{\Re +}$, $x_0, v \in \mathbb{R}^d$, $\omega, \theta \in \mathbb{R}$, we set
\begin{gather}
    G_{A_\textnormal{in}, \omega, x_0, v, \theta}^d (t,x) \coloneqq \exp \left[ i \left( \theta + 2 \lambda \omega t - v \cdot x + \frac{\abs{v}^2}{2} t \right) + \omega \right] \, u^{A_\textnormal{in}} (t, x - x_0 - v t), \label{eq:by-products_breathers}
\end{gather}
where $u^{A_\textnormal{in}}$ is the solution to \eqref{foc_log_nls} with initial data \eqref{gauss_in_data}.

\subsubsection{Nonlinear superposition}

In the context of solitons or breathers for a non-linear dispersive equation, an important question is the understanding of the interactions between them for an initial data made of the sum of several decoupled solitons or breathers. The qualitative information which come from this study should allow to better understand the dynamics and behaviours induced by the equation and its flow.
One of the usual related topics is the problem of the existence of multi-soliton solution, i.e. a solution which converges (in some sense) to the sum of solitons when $t \rightarrow \infty$, and of its stability in order to investigate whether or not they are generic objects for the dynamics of this equation.

The inverse scattering transform method was the first method used to construct multi-solitons for NLS equation \cite{Zakharov_Shabat__multisoliton_NLS_IST}. However, such a method is restricted to equations which are completely integrable (like the Korteweg-de Vries equation and the cubic nonlinear Schrödinger equation in dimension 1). Another method, using \emph{energy techniques}, meaning that it relies on the use of the second variation of the energy as a Lyapunov functional to control the difference of a solution with the soliton sum, appeared in the early 2000 \cite{Martel_Merle_Tsai__Stability_Multisolitons_gKdV} and has been used a lot since (e.g. \cite{Martel_Merle__Multi_solitary_waves_NLS, Martel_Merle_Tsai__Stability_multisoliton_NLS, Cote_Martel_Merle_Construction_multisoliton_gKdV_NLS, Cote_LeCoz_Multisolitons_NLS}).

Known results on the question of stability of multisolitary wave solutions are based on
asymptotic stability, which means that the solution converges (in some sense) as $t \rightarrow \infty$ to the sum of several solitary waves. It is relatively natural to expect that, at least when the interactions are local and the composing solitons are exponentially decaying at infinity, a multi-soliton will be orbitally stable if all the composing solitons are orbitally stable.
For some equations, like the generalized Korteweg–de Vries (gKdV) and the nonlinear Schrödinger equation (NLS), the orbital stability in $H^1$ of this sum has been shown under an assumption of flatness of the nonlinearity, and a sufficient relative speed between the solitons in the case of NLS (see for instance \cite{Martel_Merle_Tsai__Stability_Multisolitons_gKdV, Martel_Merle_Tsai__Stability_multisoliton_NLS} and the references in there).
For instance, the KdV equation is also known for having special explicit solutions, called $N$-solitons ($N \geq 2$), corresponding to the superposition of $N$ traveling waves with different speeds that interact and then remain unchanged after interaction and behaving asymptotically in large time as the sum of $N$ traveling waves (see \cite{Miura__KdV_survey}). Those $N$-solitons are also orbitally stable in $H^1$ with a more precise description of the asymptotic stability (see \cite{Martel_Merle_Tsai__Stability_Multisolitons_gKdV}).

On the other hand, as soon as one of the solitons of the sum is unstable, the multi-soliton constructed for this sum is expected to be unstable. Such a result has been proved by R. Côte and S. Le Coz \cite{Cote_LeCoz_Multisolitons_NLS} for the subcritical NLS equation with instability in $H^1$. We can also cite \cite{Chang_al__spectra_lin_op_soliton_NLS, Grillakis__Nodal_sol_semilin, Mizumachi__instability_vortex_2D_NLS} for partial results in the $L^2$-subcritical case and \cite{Grillakis__lin_instab_NLS_KG, Jones__instab_standing_wave_NLS, Mizumachi__insta_bound_states_2D_NLS, Mizumachi__vortex_2D_NLS} for results on instability with a supercritical nonlinearity.

However, even if it is unstable, a multi-soliton may still exist for this sum with (at least) one unstable soliton as soon as the relative speed is large enough (\cite{Cote_LeCoz_Multisolitons_NLS}). Thus, even though this multi-soliton should be unstable, a solution with initial data close to this multi-soliton (for instance equal to the sum of solitons) will remain close to it for some time, which should be long if the solitons in the sum get away from each other, since the interactions between them become smaller and smaller as the distance between the solitons increases (remember that the solitons vanish at infinity, and often decrease exponentially). Such a study can even be generalized to a sum of \emph{standing} solitons (meaning that the minimum relative speed is $0$), as long as they are far away from each other.

One can also wonder if breathers are stable, or also if multi-breathers exist and are stable too. Indeed, some breathers are known to be stable, for instance in $H^1$ for the mKdV equation \cite{alejo_munoz__mkdv_breather_stability_H2, Alejo_Munoz__stab_breathers_MKdV}. The question of existence and stability of multi-breathers for mKdV have also been answered in these articles (see also \cite{Chen_Liu:mKdV_soliton_resol}). Alternatively, the previous problem of \emph{nonlinear superposition} (i.e. how long a solution with a sum of breathers as initial data will remain close to this sum) seems an interesting first question in this way.

For \eqref{foc_log_nls}, we have a large class of Gaussian functions solution which includes solitons (Gaussons) and breathers. However, the only thing we know about them is that the Gausson is orbitally stable. To be able to understand further the behaviour of this equation, some numerical methods have been developed by W. Bao, R. Carles, C. Su and Q. Tang \cite{Bao_Carles_al__error_est, Bao_Carles_al__reg_num_logNLS}. In the latter, some numerical simulations have been performed and very interesting and new features have been found.
For instance, some of the behaviours found in these simulations along with the orbital stability of the Gausson suggest not only the existence of multi-solitons \cite{Ferriere__existence_multi_solitons_logNLS}, 
but also their stability.

But we will mostly focus in particular on the first two simulations in Fig. 4.5 of this article, who falls into our study with $\lambda = 1$. Beginning with two Gaussons whose distance from each other is $10$ (for the first simulation) or $6$ (for the second), the behaviour of the solution is rather different. For the former, it gives the impression that the interactions between the two Gaussons are so small that nothing seems to happen: the numerical solution has the same form at any time as the initial data and remains almost constant for a very large time. On the other hand, for the latter, the interactions between the Gaussons make them move closer from each other, very slowly at first but then faster and faster, until they cross each other at time $t=13.6$ without (almost) any change in their form, except two little structures which go to infinity on both sides. From this moment, we observe an almost periodic behaviour: the two "Gaussons" oscillate, crossing each other almost regularly, whereas some other little structures appear (less and less regularly) and go to infinity.

Therefore, a huge change in the behaviour of the solution is seen from a small change in the distance between the two Gaussons. In particular, the first simulation seems to show that the structure is rather stable, which is very surprising since no standing multi-solitons have been proved to be stable or even exist for any equation (to the best of our knowledge). Yet, the simulation may not reflect the real solution for large times, thus we will simply say that the solution remain close to the sum of the two Gaussons for a large time.

The main result of this paper is a partial result about this observation, with an estimate of the $L^2$ distance between the sum of solitons/breathers/Gaussian solutions of \eqref{foc_log_nls} and the solution of \eqref{foc_log_nls} with the previous sum as initial data.

\begin{theorem}[Nonlinear superposition principle for logNLS] \label{th_stability_stand_multi_sol}
    Let $d \in \mathbb{N}^*$. There exists $C_d > 0$ such that the following holds.
    Consider $\lambda>0$, $N \in \mathbb{N}^*$ and take $x_k \in \mathbb{R}^d$, $A_\textnormal{in}^k \in S_d (\mathbb{C})^{\Re +}$, $\omega_k \in \mathbb{R}$ and $\theta_k \in \mathbb{R}$ for $k = 1, \dots, N$ and $v \in \mathbb{R}$. Let $u$ the solution to \eqref{foc_log_nls} with initial data $u_{\textnormal{in}} (x) \coloneqq \sum G^d_{A_\textnormal{in}^k, \omega_k, x_k, v, \theta_k} (0,x)$ for any $x \in \mathbb{R}^d$. Define $A_k (t)$ provided by Proposition \ref{prop_expression_general_gaussian} for each $A_k^\textnormal{in}$ and set $G (t) \coloneqq \sum G^d_{A_k, \omega_k, x_k, v, \theta_k} (t)$ and
    \begin{equation*}
        \tau_- \coloneqq \inf_{t,k,j} \sigma (\Re A_k (t)), \qquad \tau_+ \coloneqq \sup_{t,k,j} \sigma (\Re A_k (t)),
    \end{equation*}
    where $\sigma (M)$ designates the spectrum of a matrix $M$.
    
    Then $0 < \tau_- \leq \tau_+ < \infty$ and there exist $\varepsilon_0 > 0$ depending only on $\delta \omega \coloneqq \max\limits_k |\omega_k - \omega_{k+1}|$, $\tau_-$, $\tau_+$ and $N$ such that if
    \begin{equation*}
        \varepsilon \coloneqq \left( \min\limits_{k} |x_{k+1} - x_k| \right)^{-1} < \varepsilon_0,
    \end{equation*}
    then for all $t\geq0$,
    \begin{equation*}
        \norm{u (t) - G(t)}_{L^2 (\mathbb{R}^d)} \leq C_d N^\frac{3}{2} \, \frac{\lambda \, \tau_+}{\varepsilon^{\frac{d}{2}+1} \sqrt{\tau_-}} \, \exp \left[ - \frac{\tau_-}{4 \varepsilon^2} + \max_j \omega_j + 2 \lambda t \right].
    \end{equation*}
\end{theorem}

\begin{rem}
    If the real and imaginary parts of $A_k$ commute, then Proposition \ref{prop_existence_breathers} applies and gives $\alpha_1^k, \dots, \alpha_d^k$ coming from $A_k$ for all $k = 1, \dots, N$. Hence, it is easy to prove that
    \begin{equation*}
        \tau_- = \frac{1}{2} \min_{t,k,j} r_{\alpha_j^k} (t)^{-2}, \qquad \tau_+ = \frac{1}{2} \max_{t,k,j} r_{\alpha_j^k} (t)^{-2}.
    \end{equation*}
    Furthermore, if all the $A_k$s are equal to $\lambda \, I_d$, then $\tau_- = \tau_+ = \lambda$ and the inequality becomes
    \begin{equation*}
        \norm{u (t) - \sum G^d_{A_k, \omega_k, x_k, v, \theta_k} (t)}_{L^2 (\mathbb{R}^d)} \leq C_d N^\frac{3}{2} \, \frac{\lambda^\frac{3}{2}}{\varepsilon^{\frac{d}{2}+1}} \, \exp \left[ - \frac{\lambda}{4 \varepsilon^2} + \max_j \omega_j + 2 \lambda t \right],
    \end{equation*}
    for all $t \geq 0$.
\end{rem}

This result gives an interesting time during which the solution $u$ remains close to $G \coloneqq \sum G^d_{A_k, \omega_k, x_k, v, \theta_k}$ when the minimum distance between the solitons/breathers/Gaussian solutions is large enough. Indeed, take $\delta > 0$ as small as we want. The previous result say that the inequality
\begin{equation*}
    \norm{u (t,.) - \sum G^d_{A_k, \omega_k, x_k, v, \theta_k} (t,.)}_{L^2 (\mathbb{R}^d)} \leq \delta
\end{equation*}
holds for all $t \in [0, t_\delta]$ where
\begin{equation*}
    t_\delta \coloneqq \frac{\tau_-}{8 \lambda \, \varepsilon^2} - \frac{\omega}{2 \lambda} + \frac{\frac{d}{2}+1}{2 \lambda} \ln{\varepsilon} + \frac{1}{2 \lambda} \ln{\biggl[ \frac{\delta \sqrt{\tau_-}}{C_d N^\frac{3}{2} \lambda \tau_+} \biggr]},
\end{equation*}
with $\omega \coloneqq \max_j \omega_j$, as soon as $\varepsilon$ is small enough. Thus, if we fix everything except $(x_k)_k$ and then take a sequence of family $(x_k^n)_{1 \leq k \leq N, n \in \mathbb{N}^*}$ such that
\begin{equation*}
    \varepsilon_n \coloneqq \left( \min\limits_{k} |x_{k+1}^n - x_k^n| \right)^{-1} \underset{n \rightarrow \infty}{\longrightarrow} 0,
\end{equation*}
the resulting $t_\delta^n$ can be expanded as
\begin{equation*}
    t_\delta^n = \frac{\tau_-}{8 \lambda \, \varepsilon_n^2} + \frac{\frac{d}{2}+1}{2 \lambda} \ln{\varepsilon_n} + O (1) \sim \frac{\tau_-}{8 \lambda \, \varepsilon_n^2}.
\end{equation*}
For instance, if all the $A_k$s are equal to $\lambda \, I_d$, then $\tau_- = \lambda$ and we get at first order
\begin{equation*}
    t_\delta^n \sim \frac{1}{8 \varepsilon_n^2}.
\end{equation*}
It is interesting to see that this time is in square of the minimal distance between the Gaussian functions which are in the sum. This can explain the difference we have seen in the previous two numerical examples of \cite{Bao_Carles_al__reg_num_logNLS} and the fact that a rather small change in the distance between the two Gaussons imply a bigger change in the time until which the solution remains close to the sum.

\subsection{Outline of the paper}

Section \ref{sec_gauss_data} is devoted to the study of the solution to \eqref{foc_log_nls} with Gaussian initial data. We recall in there the way to get explicit solutions (as already proved in \cite{nonlin_wave_mec, carlesgallagher, Bao_Carles_al__error_est}) which leads to Propositions \ref{prop_ODE_r} and \ref{prop_existence_breathers_d=1}, and study more carefully the behaviour of $r_\alpha$, in particular its period by proving Theorem \ref{th_period_r}. We also finish by a discussion about how the behaviour of $r_\alpha$ affects the behaviour of $u^\alpha$ (in Section \ref{subsec_discussion_sol}) and by a brief proof of Proposition \ref{prop_existence_breathers}. In Section \ref{sec_nonlin_superp}, we prove Theorem \ref{th_stability_stand_multi_sol}. The proof rely on a computation inspired from the energy estimate in $L^2$ (available thanks to Lemma \ref{lem_log_inequality}) and on Lemma \ref{lem_diff_L_log_L_sum_gaussian}, whose proof takes most of the Section.

\subsection*{Acknowledgements}

The author wishes to thank Rémi Carles and Matthieu Hillairet for enlightening and constructive discussions about this work and the writing of this paper.

\section{Propagation of Gaussian data} \label{sec_gauss_data}

In this section, we prove Propositions \ref{prop_existence_breathers_d=1} and \ref{prop_existence_breathers} and we describe more precisely the behaviour of $r_\alpha$, and in particular its period.
As already noticed in \cite{nonlin_wave_mec} and more rigorously analyzed in \cite{carlesgallagher, Bao_Carles_al__error_est}, an important feature of \eqref{foc_log_nls} is that the evolution of initial Gaussian data remains Gaussian. 
In particular, the case $d=1$ is interesting since we obtain a system of 2 ODEs which can be reduced into a single ODE.

\subsection{From \eqref{foc_log_nls} to ordinary differential equations}

We seek a solution to \eqref{foc_log_nls} (in dimension $d = 1$) under the form:
\begin{equation*}
    u(t) = b (t) \exp \Bigl( \frac{1}{2} - \frac{1}{2} \mu(t) x^2 \Bigr), \qquad t,x \in \mathbb{R},
\end{equation*}
where $\mu(t), b(t) \in \mathbb{C}$ with $\Re \mu(t) > 0$ for all $t$. We can also assume $b(0) = 1$ thanks to Remark \ref{rem_scaling}.

\begin{rem} \label{rem_sol_part_gausson_1}
    We recall the expression of the (one-dimensional) Gausson
    \begin{equation*}
        \exp \Bigl( \frac{1}{2} - \lambda x^2 \Bigr).
    \end{equation*}
    This explains why we took this form: $b (t) = 1$ and $\mu(t) = 2 \lambda$ is therefore a solution.
\end{rem}

It has already been shown in \cite{nonlin_wave_mec} that $b$ takes the form
\begin{equation*}
    b(t) = \biggl( \frac{\mu_r (t)}{\mu_r (0)} \biggr)^\frac{1}{4} e^{i \phi (t)},
\end{equation*}
where $\mu_r \coloneqq \Re \mu$ and $\phi (t)$ is given by
\begin{equation*}
    \phi (t) = - \frac{1}{2} \int_0^t \mu_r (s) \diff s + \frac{\lambda}{2} \int_0^t \ln \frac{\mu_r (s)}{\mu_r (0)} \diff s + \lambda t,
\end{equation*}
and the evolution of $\mu$ ($\mu (t) \in \mathbb{C}$) is driven by this ordinary differential equation
\begin{equation*}
    - i \dot{\mu} (t) + \mu (t)^2 = 2 \lambda \mu_r (t).
\end{equation*}
The latter can actually be rewritten in a simpler way: indeed, $\mu$ can be expressed as (see \cite{carlesgallagher})
\begin{equation}
    \mu = \frac{1}{r^2} - i \frac{\dot{r}}{r}, \label{link_mu_r}
\end{equation}
where $r$ is real and satisfies the ODE
\begin{equation}
    \Ddot{r} = \frac{1}{r^3} - \frac{2 \lambda}{r}. \label{eq_r}
\end{equation}

\begin{rem}
    The initial data of $r$ are given by the initial data of $\mu$ through \eqref{link_mu_r}, thus the two degrees of freedom can be put back on $r$. Indeed, $r(0) = ( \mu_r (0) )^{-\frac{1}{2}}$ can take any positive value whereas $\dot{r} (0) = - ( \mu_r (0) )^{-\frac{1}{2}} \, \mu_i (0)$ can take any real value independently. Thus, for any $\alpha = \alpha_r + i \alpha_i$ such that $\alpha_r > 0$ and $\alpha_i \in \mathbb{R}$, we will denote by $r_{\alpha}$ the solution to \eqref{eq_r} with initial data $r_\alpha (0) = \alpha_r$ and $\dot{r}_\alpha (0) = \alpha_i$.
\end{rem}

Hence, for any $\alpha = \alpha_r + i \alpha_i$ such that $\alpha_r > 0$ and $\alpha_i \in \mathbb{R}$ (i.e. $\alpha \in \mathbb{C}^+$),
\begin{equation*}
    u^\alpha (t) \coloneqq \sqrt{\frac{\alpha_r}{r_\alpha (t)}} \exp \Bigl( i \phi^\alpha (t) + \frac{1}{2} - \frac{1}{2 r_\alpha (t)^2} x^2 + i \frac{\dot{r}_\alpha (t)}{r_\alpha (t)} \frac{x^2}{2} \Bigr), \qquad t,x \in \mathbb{R},
\end{equation*}
where
\begin{equation*}
    \phi^\alpha (t) = - \frac{1}{2} \int_0^t \frac{1}{r_\alpha (s)^2} \diff s - \lambda \int_0^t \ln \frac{r_\alpha (s)}{\alpha_r} \diff s + \lambda t,
\end{equation*}
is solution to \eqref{foc_log_nls}.

\begin{rem}
    In particular, in the continuity of Remark \ref{rem_sol_part_gausson_1}, the Gausson \eqref{expr_gausson} is $u^\alpha$ for $\alpha = (2 \lambda)^{- \frac{1}{2}}$. Indeed, we can easily prove that $r_\alpha (t) = (2 \lambda)^{- \frac{1}{2}}$ and $\phi^\alpha (t) = 0$ for all $t \in \mathbb{R}$ with those initial data.
\end{rem}

\subsection{Study of $r_{\alpha}$}

First of all, we rescale the equation in order to make the $2 \lambda$ factor disappear. Indeed, if we define
\begin{equation*}
    \tau_{\gamma} (t) \coloneqq \sqrt{2 \lambda} \, r_\alpha \Bigl( \frac{t}{2 \lambda} \Bigr), \qquad t \in \mathbb{R},
\end{equation*}
with $\gamma \coloneqq \sqrt{2 \lambda} \alpha_r + i \frac{\alpha_i}{\sqrt{2 \lambda}}$, then $\tau_{\gamma}$ satisfies
\begin{equation}
    \Ddot{\tau}_\gamma = \frac{1}{\tau_{\gamma}^3}  - \frac{1}{\tau_{\gamma}}, \qquad
    \tau_{\gamma} (0) = \gamma_r \coloneqq \Re \gamma > 0, \qquad
    \dot{\tau}_\gamma (0) = \gamma_i \coloneqq \Im \gamma. \label{eq_tau}
\end{equation}
Thus, there remains to study $\tau_\gamma$ instead of $r_\alpha$.

First, we should prove that $\tau_\gamma$ is well defined. For any $\gamma \in \mathbb{C}^+$, the Cauchy-Lipschitz theorem gives a local definition of $\tau_\gamma$. However, since $f(x) \coloneqq x^{-3} - x^{-1} \rightarrow + \infty$ when $x \rightarrow 0^+$, we need to check that $\tau_\gamma (t)$ never touches $0$ in finite time in order to prove that $\tau_\gamma (t)$ is defined for all $t \in \mathbb{R}$.
Such a result can be proved thanks to a conserved quantity. Indeed, \eqref{eq_tau} has an Hamiltonian structure of the form:

\begin{equation*}
    \dot{q} = \frac{\partial H}{\partial p}, \qquad
    \dot{p} = - \frac{\partial H}{\partial q},
\end{equation*}
where
\begin{equation*}
    H (p, q) = \frac{1}{2} p^2 + F(q),
\end{equation*}
with
\begin{equation*}
    F (q) = \frac{1}{2 q^2} + \ln q
\end{equation*}
an anti-derivative of $f$.
In particular, $H$ is conserved by the flow of the equation, hence there holds
\begin{equation}
    E ( \tau_\gamma ) \coloneqq 2 \, H ( \dot{\tau}_\gamma, \tau_\gamma ) = ( \dot{\tau}_\gamma )^2 + \frac{1}{(\tau_\gamma)^2} + 2 \ln \tau_\gamma = E_\gamma, \label{energy_tau}
\end{equation}
where
\begin{equation}
    E_\gamma \coloneqq \gamma_i^2 + \frac{1}{(\gamma_r)^2} + 2 \ln \gamma_r. \label{def_xi}
\end{equation}
We emphasize that
\begin{equation}
    F(q) \rightarrow + \infty, \qquad \textnormal{when either} \, q \rightarrow 0 \ \textnormal{or} \ q \rightarrow + \infty. \label{div_F_dI}
\end{equation}
Hence, it is easy to prove that $\tau_\gamma (t)$ never touches $0$ (and has actually a positive lower bound) and also has an upper bound.

\begin{prop}
    For any $\gamma \in \mathbb{C}^+$, $\tau_\gamma \in \mathcal{C}^\infty ( \mathbb{R} )$ and there holds for all $t \in \mathbb{R}$
    \begin{equation*}
        \frac{1}{1 + \sqrt{E_\gamma - 1}} \leq \tau_\gamma (t) \leq \exp{\frac{E_\gamma}{2}},
    \end{equation*}
    where $E_\gamma \geq 1$ is defined by \eqref{def_xi}.
\end{prop}

\begin{proof}
    In view of the previous remark, proving the lower bound readily leads to the fact that $\tau_\gamma \in \mathcal{C}^\infty ( \mathbb{R} )$ since $f$ is $\mathcal{C}^\infty$ on $(0, \infty)$. To prove this lower bound, we use \eqref{energy_tau}. Indeed, for any $t \in I$ where $I$ is the (maximal) interval of definition for $\tau_\gamma$, there holds 
    \begin{equation*}
        2 \ln \tau_\gamma (t) = - 2 \ln \frac{1}{\tau_\gamma (t)} \geq - 2 \Bigl( \frac{1}{\tau_\gamma (t)} - 1 \Bigr),
    \end{equation*}
    where we used the fact that for all $x > 0$, $\ln x \leq x - 1$.
    Thus, plugging this inequality into \eqref{energy_tau} yields
    \begin{equation*}
        \frac{1}{(\tau_\gamma)^2} - 2 \Bigl( \frac{1}{\tau_\gamma (t)} - 1 \Bigr) \leq E_\gamma, \qquad \textnormal{i.e.} \qquad \Bigl( \frac{1}{\tau_\gamma (t)} - 1 \Bigr)^2 \leq E_\gamma - 1.
    \end{equation*}
    In particular, there also holds
    \begin{equation*}
        \frac{1}{\tau_\gamma (t)} - 1 \leq \sqrt{E_\gamma - 1},
    \end{equation*}
    and then the lower bound for $\tau_\gamma (t)$ and the fact that $\tau_\gamma$ is defined on $\mathbb{R}$ readily follow.
    On the other hand, there also holds thanks to \eqref{energy_tau}
    \begin{equation*}
        2 \ln \tau_\gamma \leq E_\gamma,
    \end{equation*}
    which leads to the upper bound.
\end{proof}

Actually, the behaviour of $\tau_\gamma$ can be better characterized. Indeed, we also emphasize that
\begin{gather*}
    f(1) = 0, \qquad
    f(x) < 0 \quad \forall x > 1, \qquad
    f(x) > 0 \quad \forall x \in (0,1).
\end{gather*}
Thus, the phase portrait (drawn in Figure \ref{phase_portrait_tau}) is rather simple and looks like that of Lotka-Volterra or prey-predator system. In particular, the trajectories are actually the level sets of $H(p,q)$. Then they describe closed Jordan curves symmetric to the x-axis that surround the point $(1,0)$ in the phase portrait, represented by
\begin{equation*}
    p = \pm \sqrt{E_\gamma - 2 F(q)} \qquad \textnormal{for} \quad \gamma_- \leq q \leq \gamma_+,
\end{equation*}
where $0 < \gamma_- < 1 < \gamma_+$ are the two solutions to the equation $2 F(q) = E_\gamma$ with unknown $q$. These values are uniquely determined because of the strict monotonicity of $F$ in the intervals $(0,1)$ and $(1, \infty)$ and $F(1) = 1 < E_\gamma$ and \eqref{div_F_dI}. Thus, similar arguments as for the Lotka-Volterra system can be applied, and then lead to the following Proposition.

\begin{figure}[htp]
\begin{center}
   \includegraphics[scale=0.44]{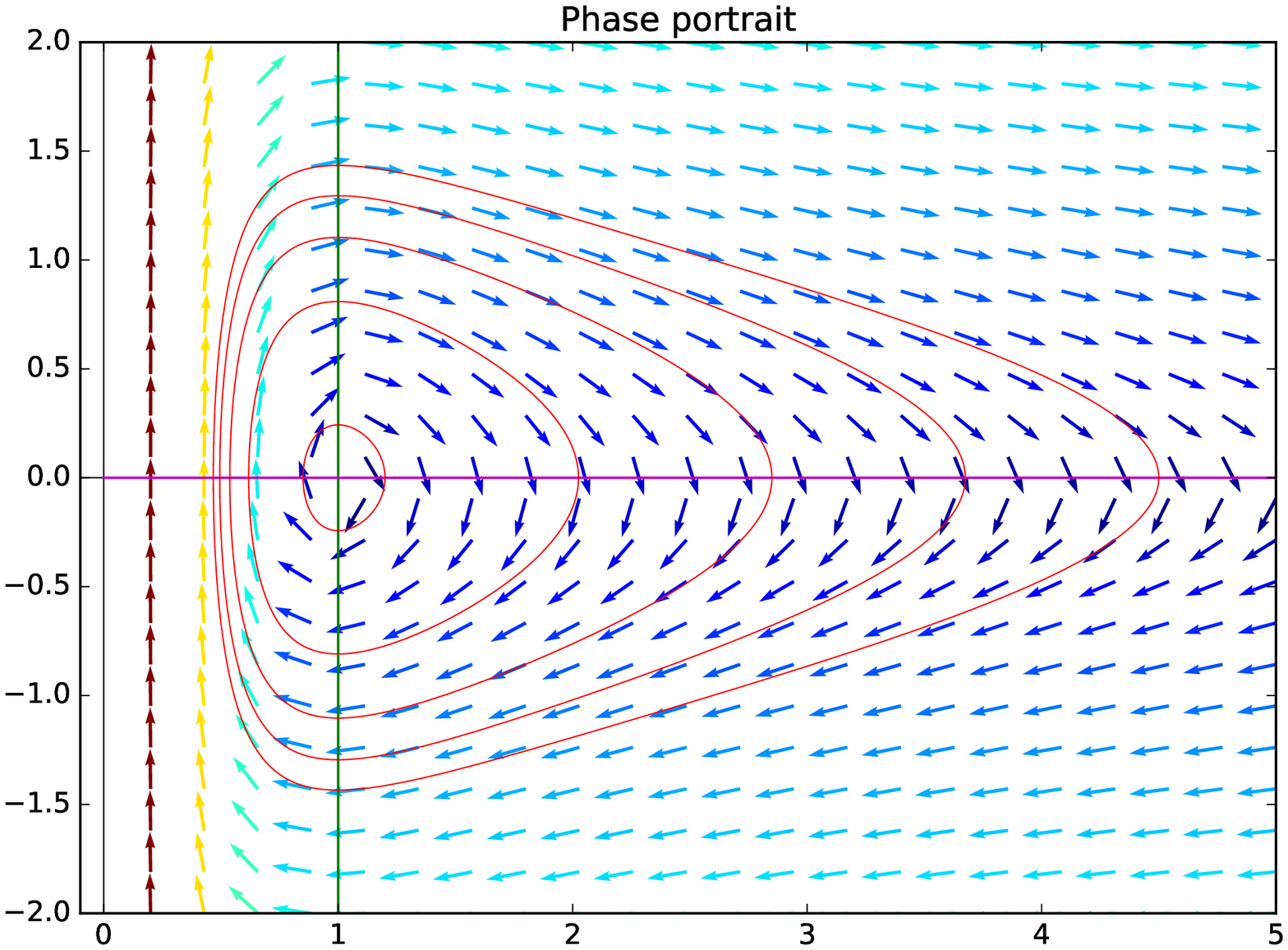}
   \includegraphics[scale=0.44]{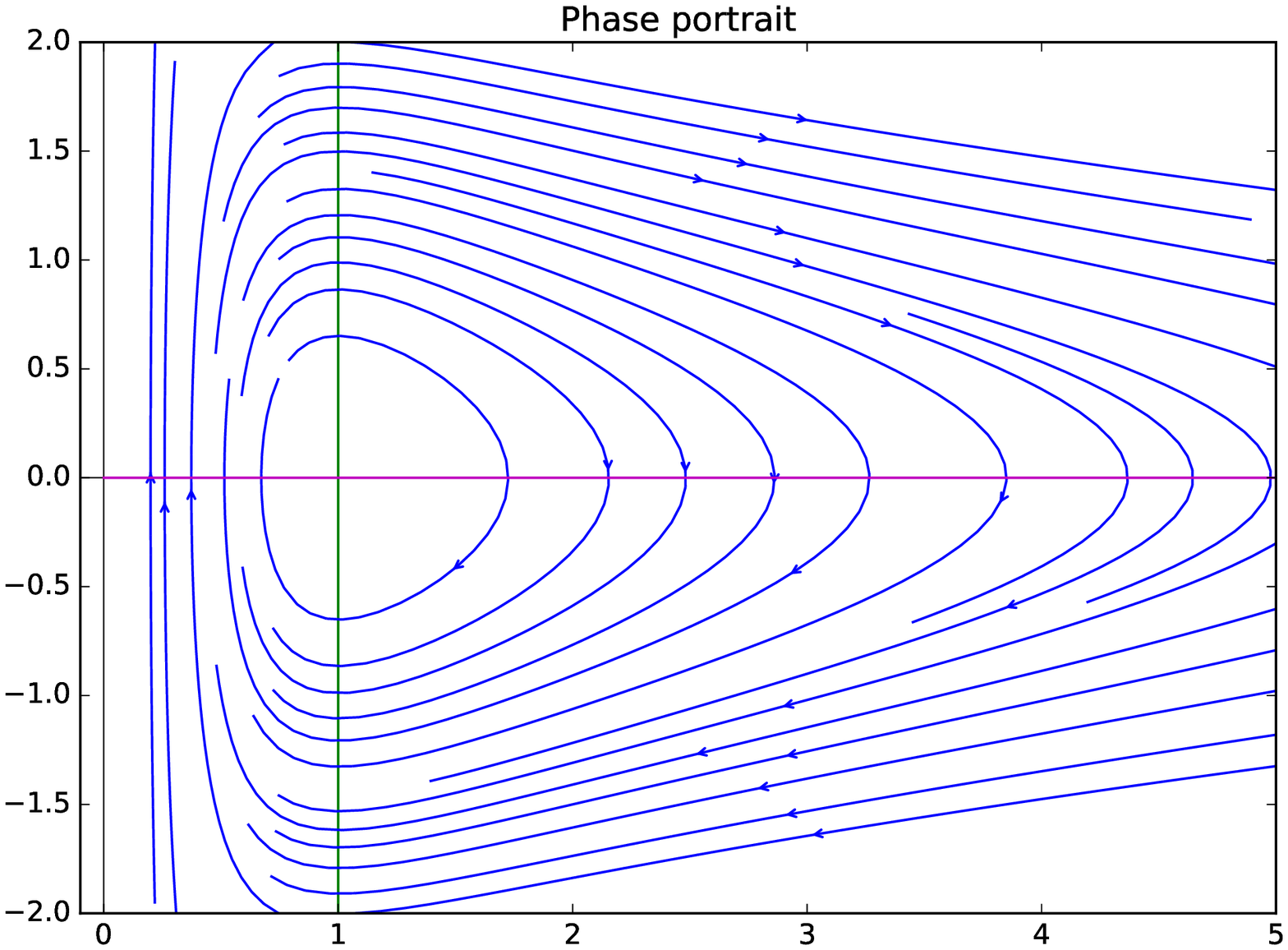}
   \caption{Phase portraits. At the left-hand side, the arrows give the direction of the flow; the green and magenta lines are the horizontal and vertical isoclines respectively; the red lines are some trajectories, which fits the level set of $H(p,q)$. At the right-hand side, a part of some trajectories are drawn along with the isoclines again.} \label{phase_portrait_tau}
\end{center}
\end{figure}

\begin{prop}
    For any $\gamma \in \mathbb{C}^+$, $\tau_\gamma$ is periodic. Moreover, for any $\gamma \in \mathbb{C}^+ \setminus \{ 1 \}$, we have $E_\gamma > 1$ and therefore the period $T_\gamma$ is given by
    \begin{equation}
        T_\gamma = 2 \int_{\gamma_-}^{\gamma_+} \frac{\diff x}{\sqrt{E_\gamma - \frac{1}{x^2} - 2 \ln x}}, \label{T_gamma}
    \end{equation}
    where $\gamma_-$ (resp. $\gamma_+$) is the only solution on $(0,1)$ (resp. $(1, \infty)$) of 
    
    \begin{equation*}
        \frac{1}{x^2} + 2 \ln x = E_\gamma.
    \end{equation*}
\end{prop}

\begin{proof}
For a more precise proof, we refer to the Chapters 3.VII. and 11.X. of \cite{Walter_EDO}, by adapting it for a smaller interval than $(- \infty, + \infty)$. This is also a re-writing of equation (6.28) of \cite{nonlin_wave_mec}.
\end{proof}

Now, we prove Theorem \ref{th_period_r}, starting by the continuity of the period with respect to the parameters. With the previous result, we know that the definition of $\gamma_-$ and $\gamma_+$ depends only on $E_\gamma > 1$, hence the period depends only on $E_\gamma$, which is continuous with respect to $\gamma \in \mathbb{C}^+$ thanks to \eqref{def_xi}.
Therefore, we only need to prove the continuity of $T_\gamma$ with respect to $E_\gamma > 1$.

\begin{prop} \label{prop_C0_periode}
    $T_\gamma$ is continuous with respect to $E_\gamma > 1$, hence also with respect to $\gamma \in \mathbb{C}^+ \setminus \{ 1 \}$.
\end{prop}

First, we shall prove the regularity of $\gamma_-$ and $\gamma_+$ with respect to $E_\gamma$ (hence also to $\gamma$).

\begin{lem}
    $\gamma_-$ and $\gamma_+$ are $\mathcal{C}^\infty$ with respect to $E_\gamma \in (1, \infty)$.
\end{lem}

\begin{proof}
    The definition of $\gamma_-$ and $\gamma_+$ leads to those two properties:
    \begin{equation*}
        \begin{cases}
            F(q) - \frac{E_\gamma}{2} = 0 \\
            q \in (0,1)
        \end{cases}
        \Longleftrightarrow q = \gamma_- (E_\gamma),
        \qquad \qquad
        \begin{cases}
            F(q) - \frac{E_\gamma}{2} = 0 \\
            q > 1
        \end{cases}
        \Longleftrightarrow q = \gamma_+ (E_\gamma).
    \end{equation*}
    Moreover, we know that $F$ is $\mathcal{C}^\infty$ on $(0, \infty)$ and that for all $q \in (0, \infty) \setminus \{ 1 \}$, there holds
    \begin{equation*}
        F'(q) = f(q) \neq 0.
    \end{equation*}
    Therefore the conclusion readily follows from the implicit function theorem.
\end{proof}

\begin{proof}[Proof of Proposition \ref{prop_C0_periode}]
    First, we cut the integral in \eqref{T_gamma} into 2 :
    \begin{equation}
        \int_{\gamma_-}^{\gamma_+} \frac{\diff x}{\sqrt{E_\gamma - \frac{1}{x^2} - 2 \ln x}} = \int_{\gamma_-}^{1} \frac{\diff x}{\sqrt{E_\gamma - \frac{1}{x^2} - 2 \ln x}} + \int_{1}^{\gamma_+} \frac{\diff x}{\sqrt{E_\gamma - \frac{1}{x^2} - 2 \ln x}}. \label{eq_int_cut}
    \end{equation}
    Thanks to this equality, we prove the continuity of $T_\gamma$ by proving the continuity of the two integrals in the right-hand side. For example, for the latter, there holds
    \begin{align}
        \int_{1}^{\gamma_+} \frac{\diff x}{\sqrt{E_\gamma - \frac{1}{x^2} - 2 \ln x}} &= \int_{1}^{(\gamma_+)^2} \frac{\diff y}{2 \sqrt{y} \sqrt{E_\gamma - \frac{1}{y} - \ln y}} \notag \\
            &= \int_{0}^{\delta_+} \frac{\diff z}{2 \sqrt{1 + z} \sqrt{E_\gamma - \frac{1}{1 + z} - \ln (1 + z)}} \notag \\
            &= \int_{0}^{1} \frac{\delta_+ \diff x}{2 \sqrt{1 + \delta_+ x} \sqrt{E_\gamma - \frac{1}{1 + \delta_+ x} - \ln (1 + \delta_+ x)}}, \label{int_cont_+}
    \end{align}
    where $\delta_+ = (\gamma_+)^2 - 1 > 0$. In particular, $\delta_+$ is continuous with respect to $E_\gamma$, therefore the integrand is continuous with respect to $E_\gamma$. Moreover, setting $g_{\delta_+} (x) = \frac{1}{1 + \delta_+ x} - \ln (1 + \delta_+ x)$, there holds
    \begin{equation*}
        g_{\delta_+} (1) = E_\gamma, \qquad
        g_{\delta_+}' (x) = (\delta_+)^2 \frac{x}{(1 + \delta_+ x)^2} \geq \frac{(\delta_+)^2}{(1 + \delta_+)^2} x \qquad \forall x \in [0,1].
    \end{equation*}
    Thus, there holds for all $x \in [0,1]$
    \begin{equation*}
        E_\gamma - g_{\delta_+} (x) \geq \frac{(\delta_+)^2}{(1 + \delta_+)^2} \int_x^1 y \diff y = \frac{(\delta_+)^2}{2 (1 + \delta_+)^2} (1-x^2) \geq \frac{(\delta_+)^2}{2 (1 + \delta_+)^2} (1-x).
    \end{equation*}
    Hence, there also holds
    \begin{equation*}
        \frac{\delta_+}{2 \sqrt{1 + \delta_+ x} \sqrt{E_\gamma - \frac{1}{1 + \delta_+ x} - \ln (1 + \delta_+ x)}} \leq \frac{1 + \delta_+}{\sqrt{2 (1-x)}}.
    \end{equation*}
    The continuity of the right-hand side of \eqref{int_cont_+} readily follows from the theorem of continuity under integral sign.
    In the same way, the first integral in the right-hand side of \eqref{eq_int_cut} can be transformed into
    \begin{align*}
        \int_{\gamma_-}^{1} \frac{\diff x}{\sqrt{E_\gamma - \frac{1}{x^2} - 2 \ln x}} &= \int_{(\gamma_-)^2}^{1} \frac{\diff y}{2 \sqrt{y} \sqrt{E_\gamma - \frac{1}{y} - \ln y}} \\
            &= \int_1^{\frac{1}{(\gamma_-)^2}} \frac{\diff z}{2 z^\frac{3}{2} \sqrt{E_\gamma - z + \ln z}} \\
            &= \int_0^{\delta_-} \frac{\diff z}{2 (1+z)^\frac{3}{2} \sqrt{E_\gamma - (1+z) + \ln (1+z)}} \\
            &= \int_0^{1} \frac{\delta_- \diff z}{2 (1+\delta_- z)^\frac{3}{2} \sqrt{E_\gamma - (1+\delta_- z) + \ln (1+\delta_- z)}},
    \end{align*}
    where $\delta_- = \frac{1}{(\gamma_-)^2} - 1 > 0$.
    Setting $h_{\delta_-} (x) \coloneqq 1 + \delta_- x - \ln (1+\delta_- z)$, there holds in the same way:
    \begin{equation*}
        h_{\delta_-} (1) = E_\gamma, \qquad h_{\delta_-}' (x) = (\delta_-)^2 \frac{x}{1 + \delta_- x} \geq \frac{(\delta_-)^2}{1 + \delta_-} x \qquad \forall x \in [0,1].
    \end{equation*}
    Hence similar arguments can be applied here and yield the continuity of the previous integral. Thus, we obtain the continuity of $T_\gamma$ with respect to $E_\gamma$, hence also with respect to $\gamma$ since $E_\gamma$ is continuous with respect to $\gamma$.
\end{proof}

Now that the continuity of the period with respect to the parameters is proved, and since it is impossible to get a simpler expression of this period, we shall find some approximations.
In particular, the question (i) of the Chapter 11.XI. of \cite{Walter_EDO} is also interesting since, along with the fact that $f'(1) = -2$, it gives the limit of the period of $\tau_\gamma$ when $\gamma$ goes to 1. We can also cite \cite{Opial_EDO} for a proof of this result.

\begin{prop}
    When $\gamma \rightarrow 1$, $T_\gamma \rightarrow \sqrt{2} \, \pi$.
\end{prop}

On the other hand, another interesting question is the period in the case of big oscillations. As one may think from the phase portrait (Fig \ref{phase_portrait_tau}) and even more from the big trajectories (Fig \ref{big_traj}), a small increase in the initial energy (the energy of the trajectories in Fig \ref{big_traj} goes from $4$ to $6$) induces a big increase for the maximum of $\tau_\gamma$ (from $20$ to $50$) and for the period. To prove this behaviour, explicit computations and inequalities are required and yield the following result.

\begin{figure}[htp]
\begin{center}
   \includegraphics[scale=0.6]{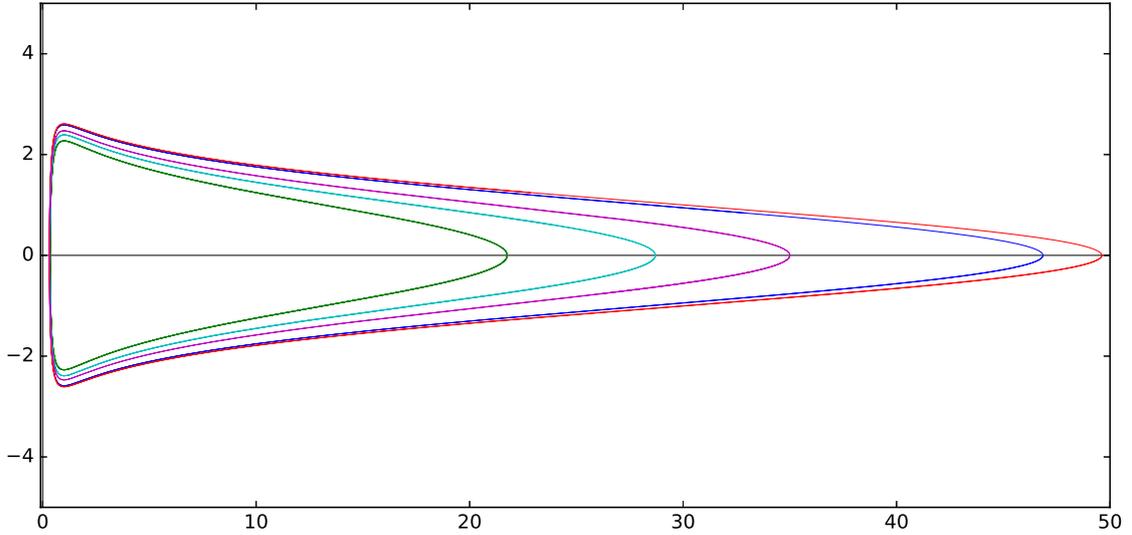}
   \caption{Plot of some trajectories of $\tau_\gamma$ in the phase space for 5 real $\gamma$ between $10$ and $50$.} \label{big_traj}
\end{center}
\end{figure}

\begin{prop} \label{prop_period_inf}
    When $\gamma \rightarrow \infty$ or $\Re \gamma \rightarrow 0$,
    \begin{equation*}
        T_\gamma \sim \sqrt{2 \pi} \exp{\frac{E_\gamma}{2}}.
    \end{equation*}
\end{prop}

\begin{proof}
    First, we emphasize that the condition $\gamma \rightarrow \infty$ or $\Re \gamma \rightarrow 0$ is equivalent to the simpler condition $E_\gamma \rightarrow + \infty$, and therefore also to the facts that $\gamma_- \rightarrow 0$ and $\gamma_+ \rightarrow + \infty$. To be more precise 
    %
    for $\gamma_+$, there holds
    \begin{equation*}
        2 \ln \gamma_+ + o(1) = E_\gamma, \qquad \textnormal{i.e.} \qquad \gamma_+ \sim \exp{\frac{E_\gamma}{2}}.
    \end{equation*}

    Then we cut the integral in \eqref{T_gamma} in two: before $1$ and after $1$. First,
    \begin{align*}
        \int_{\gamma_-}^{1} \frac{\diff x}{\sqrt{E_\gamma - \frac{1}{x^2} - 2 \ln x}} &= \int_{\gamma_-}^1 \frac{\diff x}{\sqrt{\frac{1}{(\gamma_-)^2} - \frac{1}{x^2} - 2 \ln \frac{x}{\gamma_-}}} \\
        &= \gamma_- \int_1^{\frac{1}{\gamma_-}} \frac{\diff y}{\sqrt{\frac{1}{(\gamma_-)^2} \Bigl( 1 - \frac{1}{y^2} \Bigr) - 2 \ln y}} \\
        &= (\gamma_-)^2 \int_1^{\frac{1}{\gamma_-}} \frac{\diff y}{\sqrt{1 - \frac{1}{y^2} - 2 (\gamma_-)^2 \ln y}} \\
        &\leq (\gamma_-)^2 \int_1^{\frac{1}{\gamma_-}} \frac{\diff y}{\sqrt{1 - \frac{1}{y^2} - 2 (\gamma_-)^2 (y - 1)}} \\
        &\leq (\gamma_-)^2 \int_1^{\frac{1}{\gamma_-}} \frac{\diff y}{\sqrt{\bigl(\frac{y + 1}{y^2} - 2 (\gamma_-)^2 \bigr) (y - 1)}}.
    \end{align*}
    Moreover, there holds for all $y \in [1, \frac{1}{\gamma_-}]$
    \begin{equation*}
        \frac{y + 1}{y^2} \geq \gamma_- + (\gamma_-)^2,
    \end{equation*}
    so that
    \begin{align*}
        \int_{\gamma_-}^{1} \frac{\diff x}{\sqrt{E_\gamma - \frac{1}{x^2} - 2 \ln x}} &\leq \frac{(\gamma_-)^2}{\sqrt{\gamma_- - (\gamma_-)^2}} \int_1^{\frac{1}{\gamma_-}} \frac{\diff y}{\sqrt{y - 1}} \\
        &\leq \frac{\gamma_-}{\sqrt{\frac{1}{\gamma_-} - 1}} \Bigl[ \frac{1}{2} \sqrt{y-1} \Bigr]_1^\frac{1}{\gamma_-} \\
        &\leq \frac{\gamma_-}{2} \longrightarrow 0.
    \end{align*}

    On the other hand, we will prove a lower and an upper bound for the second part of the integral. For the lower bound, we recall that $2 F(\gamma_+) = E_\gamma$, so for all $x \in [1, \gamma_+]$
    \begin{equation*}
        E_\gamma - \frac{1}{x^2} = \frac{1}{(\gamma_+)^2} - \frac{1}{x^2} + 2 \ln \gamma_+ \leq 2 \ln \gamma_+.
    \end{equation*}
    Therefore,
    \begin{align}
        \int_1^{\gamma_+} \frac{\diff x}{\sqrt{E_\gamma - \frac{1}{x^2} - 2 \ln x}} &\geq \int_1^{\gamma_+} \frac{\diff x}{\sqrt{2 \ln \gamma_+ - 2 \ln x}} \notag \\
            &\geq \gamma_+ \int_{\frac{1}{\gamma_+}}^{1} \frac{\diff y}{\sqrt{- 2 \ln y}}. \label{low_bound_period}
    \end{align}
    The last integral converges as $\gamma_+ \rightarrow \infty$ to
    \begin{equation*}
        \int_{0}^{1} \frac{\diff y}{\sqrt{- 2 \ln y}} = \int_0^{\infty} \frac{e^{-z} \diff z}{\sqrt{2 z}} = \int_0^\infty \sqrt{2} \, e^{- \zeta^2} \diff \zeta = \sqrt{\frac{\pi}{2}}.
    \end{equation*}
    Hence, the right-hand side of \eqref{low_bound_period} is equivalent to
    \begin{equation*}
        \sqrt{\frac{\pi}{2}} \exp{\frac{E_\gamma}{2}}.
    \end{equation*}
    
    For the upper bound, with a change of variables $x = e^y$ and with $y_+ = \ln \gamma_+ \rightarrow + \infty$, we first obtain
    \begin{equation*}
        \int_1^{\gamma_+} \frac{\diff x}{\sqrt{E_\gamma - \frac{1}{x^2} - 2 \ln x}} = \int_0^{y_+} \frac{e^y \diff y}{\sqrt{E_\gamma - e^{-2y} - 2 y}}.
    \end{equation*}
    Moreover, using the convexity of $y \mapsto e^{-2y}$, there holds for all $y \in [0, y_+]$
    \begin{equation*}
        e^{-2y} \leq 1 - \frac{y}{y_+} (1 - e^{-2 y_+}).
    \end{equation*}
    Thus, using also the fact that $E_\gamma = e^{-2 y_+} + 2 y_+$, we obtain
    \begin{align*}
        \int_1^{\gamma_+} \frac{\diff x}{\sqrt{E_\gamma - \frac{1}{x^2} - 2 \ln x}} &\leq \int_0^{y_+} \frac{e^y \diff y}{\sqrt{E_\gamma - \bigl( 1 - \frac{y}{y_+} (1 - e^{-2 y_+}) \bigr) - 2 y}} \\
            &\leq \int_0^{y_+} \frac{e^y \diff y}{\sqrt{\Bigl( 2 - \frac{1 - e^{-2 y_+}}{y_+} \Bigr) (y_+ - y)}} \\
            &\leq \frac{e^{y_+}}{\sqrt{2 - \frac{1 - e^{-2 y_+}}{y_+}}} \int_0^{y_+} \frac{e^{-z} \diff z}{\sqrt{z}} \\
            &\leq \frac{\gamma_+}{\sqrt{2 - \frac{1 - e^{-2 y_+}}{y_+}}} \int_0^{(y_+)^2} 2 e^{-z^2} \diff z \ \sim \sqrt{\frac{\pi}{2}} \exp{\frac{E_\gamma}{2}}.
    \end{align*}
    The conclusion readily follows.
\end{proof}

\subsection{Discussion on $u^\alpha$} \label{subsec_discussion_sol}

Those results show an interesting and surprising feature: the "period" of $u^\alpha$ (i.e. the period of $r_\alpha$) can be very large. Moreover, if we take a Gaussian initial data very concentrated
\begin{equation*}
    u_0 (x) = \exp \Bigl( \frac{1}{2} - \delta x^2 \Bigr)
\end{equation*}
with $\delta > 0$ large compared to $\lambda$, the solution will first disperse, very quickly at the beginning but more and more slowly, until a time when this behaviour turns round. Then the solution will re-concentrate, slowly at first and more and more quickly until it comes back to its initial value (up to a complex modulation) at a time around
\begin{equation*}
    \sqrt{\frac{\pi}{2}} \exp{\frac{E_\gamma}{2}} = \sqrt{\frac{\pi \lambda}{2 \delta} } \exp{\frac{\delta}{2 \lambda}}.
\end{equation*}
Indeed, in the proof of Proposition \ref{prop_period_inf}, we also proved implicitly that $\tau_\gamma$ is most of the time larger than $1$ when $E_\gamma$ is large. Even more, we proved that the time during which $\tau_\gamma$ is less than $1$, given by 
\begin{equation*}
    2 \int_{\gamma_-}^{1} \frac{\diff x}{\sqrt{E_\gamma - \frac{1}{x^2} - 2 \ln x}},
\end{equation*}
goes in fact to $0$ as $E_\gamma$ goes to $\infty$.

Since such very flat initial data give a solution which remains flat for a long time but then has a high "peak" which suddenly appears for a brief time, this behaviour might be related to rogue waves.
For instance, the absolute value of the breather with $\delta = 35$ and $\lambda = 0.5$ in the initial data is plotted in Figure \ref{plot_breather}.

\begin{figure}[htp]
\begin{center}
   \includegraphics[scale=0.6]{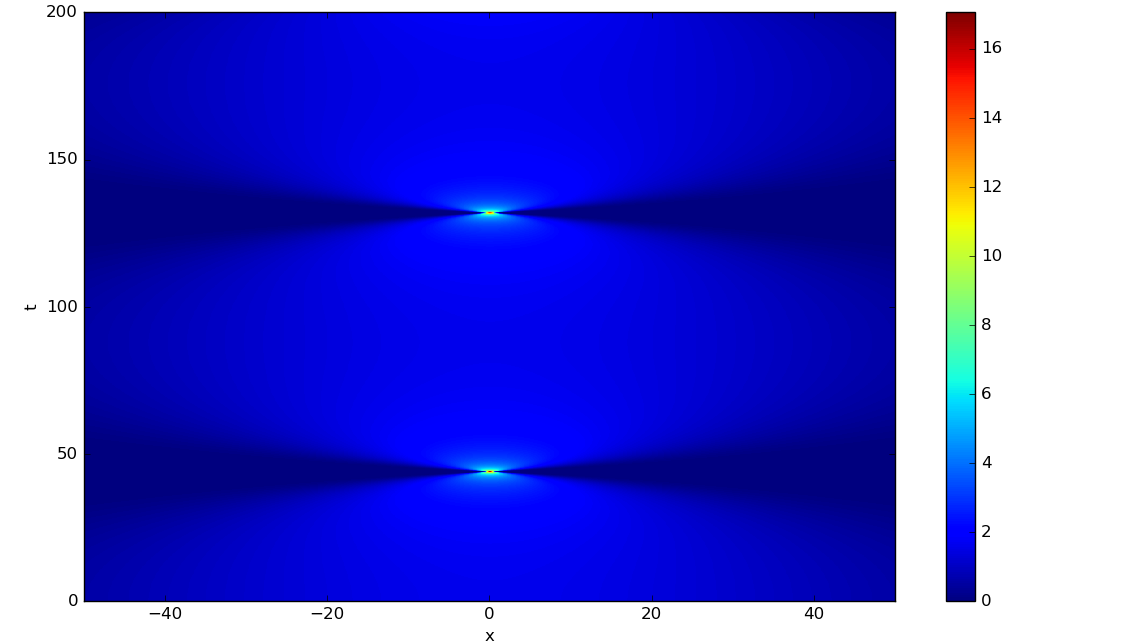}
   \caption{Plot of the breather ($\lambda = 0.5$) with initial data $\exp \Bigl( \frac{1}{2} - \delta x^2 \Bigr)$ with $\delta^{-1} = 2 \times 35^2$. The first "peak" is at $t = 43.86$. \label{plot_breather}}
\end{center}
\end{figure}

\subsection{Proof of Proposition \ref{prop_existence_breathers}}

The proof of this proposition relies on the following result which characterizes $S_d (\mathbb{C})^{\Re ++}$.

\begin{lem} \label{lem_M_d_C+}
    For any $A \in S_d (\mathbb{C})^{\Re ++}$, there exists $R \in \mathcal{O}_d (\mathbb{R})$ and $\beta_1, \dots, \beta_d \in \mathbb{C}^+$ such that
    \begin{equation*}
        A = R D R^\top,
    \end{equation*}
    where
    \begin{equation*}
          D =
              \begin{bmatrix}
                \beta_{1} & & \\
                & \ddots & \\
                & & \beta_{d}
              \end{bmatrix}.
    \end{equation*}
\end{lem}

\begin{proof}
    By definition, $A_r \coloneqq \Re A$ and $A_i \coloneqq \Im A$ are real symmetric and commute. Therefore, they can be orthogonally co-diagonalized, which means that we have a $R \in \mathcal{O}_d (\mathbb{R})$ and $D_r$ and $D_i$ real diagonal matrices such that $A_r = R D_r R^\top$ and $A_i = R D_i R^\top$. Moreover, since $A_r$ is positive definite, the diagonal coefficients of $D_r$ are positive. Therefore, $A = R D R^\top$ where $D = D_r + i D_i$ is a diagonal matrix whose diagonal coefficients are in $\mathbb{C}^+$.
\end{proof}

\begin{proof}[Proof of Proposition \ref{prop_existence_breathers}]
    Set $R \in \mathcal{O}_d (\mathbb{R})$, $\beta_1, \dots, \beta_d \in \mathbb{C}^+$ and $D$ given by Lemma \ref{lem_M_d_C+} for $A$. Since $u^A$ is a solution to \eqref{foc_log_nls} which is invariant under an orthogonal transformation, $u^A (t, Rx)$ is also solution to \eqref{foc_log_nls} with initial data
    \begin{equation*}
        u^A_\textnormal{in} (Rx) = \exp{\Bigl[ \frac{d}{2} - x^\top D x \Bigr]} = u^{\alpha_1}_\textnormal{in} (x_1) \dots u^{\alpha_d}_\textnormal{in} (x_d),
    \end{equation*}
    where for all $j \in \{ 1, \dots, d \}$
    \begin{equation*}
        \alpha_j \coloneqq \frac{1}{\sqrt{2 \Re \beta_j}} - i \frac{\Im \beta}{\sqrt{2 \Re \beta_j}}
    \end{equation*}
    Since every $u^{\alpha_j}$ is solution to \eqref{foc_log_nls} in dimension $1$, we know that
    \begin{equation*}
        u^{\alpha_1} (t) \otimes \dots \otimes u^{\alpha_d} (t)
    \end{equation*}
    is solution to \eqref{foc_log_nls} with the same previous initial data. Thus, by uniqueness of the solution in $\mathcal{C}_b ( \mathbb{R}, W (\mathbb{R}^d))$, there holds
    \begin{equation*}
        u^A (t, R \, . ) = u^{\alpha_1} (t) \otimes \dots \otimes u^{\alpha_d} (t). \qedhere
    \end{equation*}
\end{proof}

\section{Nonlinear superposition} \label{sec_nonlin_superp}

In this section, we prove Theorem \ref{th_stability_stand_multi_sol} (in any dimension $d \in \mathbb{N}^*$). This result is directly inspired from the energy estimate in $L^2$ found in \cite{cazenave-haraux} to prove the uniqueness of the solution for \eqref{foc_log_nls} in the case $\lambda < 0$. 
This energy estimate is the consequence of the following lemma:

\begin{lem}[{\cite[Lemma~1.1.1]{cazenave-haraux}}] \label{lem_log_inequality}
There holds

\begin{equation*}
    \abs{\Im \left( (z_2 \ln \abs{z_2}^2 - z_1 \ln \abs{z_1}^2) (\overline{z_2} - \overline{z_1}) \right)} \leq 2 \abs{z_2 - z_1}^2, \qquad \forall z_1, z_2 \in \mathbb{C}.
\end{equation*}
\end{lem}

Indeed, taking $u_1$ and $u_2$ two solutions to \eqref{foc_log_nls}, $u \coloneqq u_1 - u_2$ satisfies
\begin{equation*}
    i \, \partial_t u + \frac{1}{2} \Delta u = - \lambda \Bigl( u_1 \ln{\abs{u_1}^2} - u_2 \ln{\abs{u_2}^2} \Bigr).
\end{equation*}
Thus, we directly get
\begin{equation*}
    \frac{1}{2} \frac{\diff}{\diff t} \norm{u (t)}_{L^2}^2 = - \lambda \Im \int \Bigl( u_1 \ln{\abs{u_1}^2} - u_2 \ln{\abs{u_2}^2} \Bigr) ( \overline{u_1} - \overline{u_2} ) \diff x \leq 2 \abs{\lambda} \, \norm{u (t)}_{L^2}^2.
\end{equation*}
We emphasize that this inequality does not involve the $H^1$ norm of $u_1$ or $u_2$: it only involves the $L^2$ norm of $u$.
On the other hand, if $v$ is solution to \eqref{foc_log_nls} with initial data $u_1 (0) + u_2 (0)$, one can wonder how close $v (t)$ will stay to $u(t) \coloneqq u_1 (t) + u_2 (t)$. If now we set $w \coloneqq v - (u_1 - u_2) = v - u$, then it satisfies
\begin{equation*}
    i \, \partial_t w + \frac{1}{2} \Delta w = - \lambda \Bigl( v \ln{\abs{v}^2} - u_1 \ln{\abs{u_1}^2} - u_2 \ln{\abs{u_2}^2} \Bigr).
\end{equation*}
Therefore, there holds
\begin{align*}
    \frac{1}{2} \frac{\diff}{\diff t} \norm{w (t)}_{L^2}^2 &= - \lambda \Im \int \Bigl( v \ln{\abs{v}^2} - u_1 \ln{\abs{u_1}^2} - u_2 \ln{\abs{u_2}^2} \Bigr) ( \overline{v} - \overline{u} ) \diff x \\
        &\begin{multlined}[t][13cm] = - \lambda \Im \int \Bigl( v \ln{\abs{v}^2} - u \ln{\abs{u}^2} \Bigr) ( \overline{v} - \overline{u} ) \diff x \\ - \lambda \Im \int \Bigl( u \ln{\abs{u}^2} - u_1 \ln{\abs{u_1}^2} - u_2 \ln{\abs{u_2}^2} \Bigr) ( \overline{v} - \overline{u} ) \diff x \end{multlined} \\
    \frac{1}{2} \abs{\frac{\diff}{\diff t} \norm{w (t)}_{L^2}^2} &\leq 2 \abs{\lambda} \, \norm{w}_{L^2}^2 + \abs{\lambda} \int \abs{ u \ln{\abs{u}^2} - u_1 \ln{\abs{u_1}^2} - u_2 \ln{\abs{u_2}^2} } \abs{w} \diff x \\
        &\leq 2 \abs{\lambda} \, \norm{w}_{L^2}^2 + \abs{\lambda} \norm{ u \ln{\abs{u}^2} - u_1 \ln{\abs{u_1}^2} - u_2 \ln{\abs{u_2}^2} }_{L^2} \norm{w}_{L^2}.
\end{align*}
Dividing by $\norm{w}_{L^2}$, we obtain
\begin{equation*}
    \abs{\frac{\diff}{\diff t} \norm{w (t)}_{L^2}} \leq 2 \abs{\lambda} \, \norm{w}_{L^2} + \abs{\lambda} \norm{ u \ln{\abs{u}^2} - u_1 \ln{\abs{u_1}^2} - u_2 \ln{\abs{u_2}^2} }_{L^2}.
\end{equation*}

This estimate can also be generalized to more than 2 solutions with the same computation: for any integer $N \geq 2$ and any solutions $v, u_1, \dots, u_N$ to \eqref{foc_log_nls}, the function $w \coloneqq v - u$ where $u \coloneqq \sum u_j$ satisfies the inequality
\begin{equation}
    \abs{\frac{\diff}{\diff t} \norm{w (t)}_{L^2}} \leq 2 \abs{\lambda} \, \norm{w}_{L^2} + \abs{\lambda} \norm{ u \ln{\abs{u}^2} - \sum_{j=1}^N u_j \ln{\abs{u_j}^2} }_{L^2}. \label{energy_est_like}
\end{equation}
This estimate can be useful up to two conditions. First, we must know a time $t_0$ where $v (t_0)$ and $u (t_0)$ are close in $L^2$. Then, we also need a way to estimate the last term in the right-hand side. This term should be small for instance if the $u_i$s are "well separated". Such a thing may be hard to prove in general, but it is easier if we have an explicit expression for the $u_i$s. This is the case for the breathers and Gaussons, or more generally for the Gaussian functions solution. In particular, if their centers are far away from each other, then this term is actually very small:

\begin{lem} \label{lem_diff_L_log_L_sum_gaussian}
    For any $d \in \mathbb{N}^*$, there exists $C_d > 0$ such that the following holds.
    Let $N \in \mathbb{N}^*$ and take $x_k \in \mathbb{R}^d$, $\omega_k \in \mathbb{R}$, $\Lambda_k \in S_d (\mathbb{C})^{\Re +}$ and $\theta_k : \mathbb{R}^d \rightarrow \mathbb{R}$ a real measurable function for $k = 1, \dots, N$, and define for all $x \in \mathbb{R}^d$
    \begin{equation*}
        g_k (x) = \exp \left[ i \theta_k (x) + \omega_k - (x - x_k)^\top \Lambda_k (x - x_k) \right],
    \end{equation*}
    as well as
    \begin{equation*}
        g (x) = \sum_{k = 1,\dots,N} g_k (x).
    \end{equation*}
    If
    \begin{equation*}
        \varepsilon \coloneqq \left( \min_{k \neq j} \, \abs{x_{j} - x_k} \right)^{-1} < \varepsilon_0 \coloneqq  \min \biggl( \frac{\sqrt{\lambda_+}}{\max (\sqrt{\delta \omega + 1}, \sqrt{\ln{N}})}, \sqrt{\frac{\lambda_-}{d+2}} \biggr)
    \end{equation*}
    where $\delta \omega \coloneqq \max\limits_{j,k} |\omega_k - \omega_j|$, $\lambda_+ = \max\limits_k \Re \sigma (\Lambda_k)$ and $\lambda_- = \min\limits_k \Re \sigma (\Lambda_k) > 0$, then
    \begin{equation*} 
        \norm{g \ln \abs{g} - \sum_{k = 1}^N g_k \ln \abs{g_k}}_{L^2 (\mathbb{R}^d)} \leq C_d N^\frac{3}{2} \, \frac{\lambda_+}{\varepsilon^{\frac{d}{2}+1} \sqrt{\lambda_-}} \, \exp \left[ - \frac{\lambda_-}{4 \varepsilon^2} + \max_j \omega_j \right].
    \end{equation*}
\end{lem}

Such an estimate allows us to prove Theorem \ref{th_stability_stand_multi_sol}.

\begin{proof}[Proof of Theorem \ref{th_stability_stand_multi_sol}]
    Thanks to Proposition \ref{prop_expression_general_gaussian} and \eqref{eq:by-products_breathers}, we know that each \\ $G_k \coloneqq G^d_{A_k, \omega_k, x_k, v, \theta_k}$ can be written under the form
    \begin{equation*}
        G_k (t,x) = \exp \left[ i \theta_k (t,x) + \tilde{\omega}_k (t) - (x - x_k - v t)^\top A_k (t) (x - x_k - v t) \right],
    \end{equation*}
    with
    \begin{equation*}
        \tilde{\omega}_k (t) = \omega_k + \frac{d}{2} - \frac{1}{4} \ln{\frac{\det{\Re A_k (t)}}{\det{\Re A_k (0)}}}.
    \end{equation*}
    In particular, we have 
    \begin{equation*}
        \sup_{t,\ell,k} \, \abs{\tilde{\omega}_\ell (t) - \tilde{\omega}_k (t)} \leq \delta \tilde{\omega} \coloneqq \delta \omega + \frac{d}{2} \ln{\frac{\tau_+}{\tau_-}},
    \end{equation*}
    where $\tau_+ = \sup\limits_{t,k} \Re \sigma (A_k (t))$ and $\tau_- = \inf\limits_{t,k} \Re \sigma (A_k (t))$.
    Hence, setting $G \coloneqq \sum_k G_k$ and
    \begin{equation*}
        \varepsilon_0 \coloneqq  \frac{1}{\sqrt{2}} \min \biggl( \frac{\sqrt{\tau_+}}{\max (\sqrt{\delta \tilde{\omega} + 1}, \sqrt{\ln{N}})}, \sqrt{\frac{\tau_-}{\frac{d}{2} + 1}} \biggr),
    \end{equation*}
    and since $\abs{x_k - vt - (x_j - vt)} = \abs{x_k - x_j} \geq \varepsilon^{-1}$ for all $k \neq j$, there holds from Lemma \ref{lem_diff_L_log_L_sum_gaussian}
    \begin{equation*}
        \norm{G \ln \abs{G} - \sum_{k = 1}^N G_k \ln \abs{G_k}}_{L^2 (\mathbb{R}^d)} \leq C_d N^\frac{3}{2} \, \frac{\tau_+}{\varepsilon^{\frac{d}{2}+1} \sqrt{\tau_-}} \, \exp \left[ - \frac{\tau_-}{4 \varepsilon^2} + \max_j \omega_j \right]
    \end{equation*}
    as soon as $\varepsilon < \varepsilon_0$.
    Plugging this into \eqref{energy_est_like}, we get
    \begin{equation*}
        \frac{\diff}{\diff t} \norm{w (t)}_{L^2} \leq 2 \abs{\lambda} \, \norm{w}_{L^2} + C_d N^\frac{3}{2} \, \frac{\lambda \, \tau_+}{\varepsilon^{\frac{d}{2}+1} \sqrt{\tau_-}} \, \exp \left[ - \frac{\tau_-}{4 \varepsilon^2} + \max_j \omega_j \right],
    \end{equation*}
    where $w = u - G$. The result readily follows from the Gronwall lemma and the fact that $w(0) = 0$.
\end{proof}

\subsection{Proof of Lemma \ref{lem_diff_L_log_L_sum_gaussian}}

This lemma shows that we can approximate the non-linearity in the equation for the sum of the Gaussian functions by the sum of the non-linearity of each Gaussian, as soon as these Gaussian are well separated. This kind of result is rather usual, as we often find it when talking about multi-solitons for instance.
The proof usually uses the exponential decay at infinity of the solitons along with the fact that the non-linearity is locally Lipschitz.
Here, we have a better decay at infinity for our functions (which are Gaussian), but our non-linearity $F(z) \coloneqq z \ln \abs{z}^2$ is not Lipschitz at $0$, therefore this kind of result is not obvious at first sight. However, $F$ is actually almost Lipschitz in the following sense:

\begin{lem} \label{lem_lip_z_log_z}
    For all $z, \tilde{z} \in \mathbb{C}$ such that $\abs{z} \leq 1$, $\abs{\tilde{z}} \leq 1$ and $z \neq 0$, there holds
    
    \begin{equation*}
        \abs{F(\tilde{z}) - F(z)} \leq \abs{z - \tilde{z}} \Bigl[ 6 - \ln \abs{z}^2 \Bigr].
    \end{equation*}
\end{lem}

This lemma is interesting: for any $x, \tilde{x} \in (0,1]$, the mean value inequality would only give
\begin{equation*}
    \abs{F (x) - F (\tilde{x})} \leq \abs{x - \tilde{x}} \Bigl[ 6 - 2 \ln \min (x, \tilde{x}) \Bigr].
\end{equation*}
In particular, if we fix $x$ and make $\tilde{x}$ goes to $0$, the right-hand side goes to $\infty$ whereas the left-hand side only goes to $F(x) = x \ln{x^2}$. Therefore the above inequality is not optimal and not fitted when we take $\tilde{x}$ which may be very small compared to $x$ or even sometimes vanish.
This lemma shows that we can actually take either $\ln{x}$ or $\ln{\tilde{x}}$ in the right-hand side without having to know which one of $x$ or $\tilde{x}$ is the smallest.
Another advantage is that the expression of any $\ln{\abs{g_j (x)}^2}$ is clearly simpler than $\ln{\abs{g (x)}^2}$, which will allow us to have simpler computations when applying the previous lemma with $z = g_j (x)$ and $\tilde{z} = g (x)$.

\begin{proof}
    For this proof only, we use the identification $\mathbb{C} \approx \mathbb{R}^2$, and we see $F$ as a function from $\mathbb{R}^2$ to itself :
    \begin{equation*}
        F : z = \begin{bmatrix}
            z_r \\
            z_i
            \end{bmatrix} \in \mathbb{R}^2 \mapsto z \ln \abs{z}^2 = \begin{bmatrix}
            z_r \ln \abs{z}^2 \\
            z_i \ln \abs{z}^2
            \end{bmatrix}.
    \end{equation*}
    Then, $F$ is differentiable on $\mathbb{C} \setminus \{ 0 \}$ and we can compute for $z \neq 0$
    \begin{equation*}
        DF (z) = \begin{bmatrix}
            \ln \abs{z}^2 + 2 \frac{z_r^2}{\abs{z}^2} & 2 \frac{z_r z_i}{\abs{z}^2} \\
             2\frac{z_r z_i}{\abs{z}^2} & \ln \abs{z}^2 + 2 \frac{z_i^2}{\abs{z}^2}
            \end{bmatrix} = R_z^{-1} \begin{bmatrix}
            \ln \abs{z}^2 + 2 & 0 \\
            0 & \ln \abs{z}^2
            \end{bmatrix} R_z,
    \end{equation*}
    where $R_z$ is the rotation which maps $z$ onto the real positive half-line. Hence, there holds for all $z \in \mathbb{C}^*$
    \begin{equation*}
        \norm{DF (z)}_2 \leq 2 (\abs{\ln \abs{z}} + 1).
    \end{equation*}
    Then, we compute for $z$ and $\tilde{z}$ satisfying the assumptions above:
    \begin{align}
        F(\tilde{z}) - F(z) &= \int_0^1 DF (z + t (\tilde{z} - z) ) \, (\tilde{z} - z) \diff t, \notag \\
        \abs{F(\tilde{z}) - F(z)} &\leq \int_0^1 \norm{DF (z + t (\tilde{z} - z) )}_2 \diff t \, \abs{\tilde{z} - z} \notag \\
            &\leq 2 \int_0^1 \Bigl( \abs*[\Big]{\ln \abs{z + t (\tilde{z} - z)}} + 1 \Bigr) \diff t \, \abs{\tilde{z} - z}. \label{comp1}
    \end{align}
    By assumption, there holds:
    \begin{equation*}
        \abs*[\Big]{\abs{z} - t \abs{\tilde{z} - z}} \leq \abs{z + t (\tilde{z} - z)} \leq (1-t) \abs{z} + t \abs{\tilde{z}} \leq 1.
    \end{equation*}
    Since $y \mapsto \ln y$ is increasing and non-positive on $(0,1]$, we have for a.e. $t \in [0,1]$
    \begin{equation*}
        \abs*[\Big]{\ln \abs{z + t (\tilde{z} - z)}} = - \ln \abs{z + t (\tilde{z} - z)} \leq - \ln \abs*[\Big]{\abs{z} - t \abs{\tilde{z} - z}}.
    \end{equation*}
    Putting this in \eqref{comp1}, we get
    \begin{align*}
        \abs{F(\tilde{z}) - F(z)} &\leq 2 \int_0^1 ( 1 - \ln \abs*[\Big]{\abs{z} - t \abs{\tilde{z} - z}} ) \diff t \, \abs{\tilde{z} - z} = 2 \int_{\abs{z} - \abs{\tilde{z} - z}}^{\abs{z}} (1 - \ln \abs{v}) \diff v \\
            &\leq 2 \Bigl[ 2v - v \ln \abs{v} \Bigr]_{\abs{z} - \abs{\tilde{z} - z}}^{\abs{z}} = 4 \abs{\tilde{z} - z} + 2 (\abs{z} - \abs{\tilde{z} - z}) \ln \abs*[\Big]{\abs{z} - \abs{\tilde{z} - z}} - 2 \abs{z} \ln \abs{z}.
    \end{align*}
    Then, we need to estimate the difference between the two last terms with the following lemma.
    \begin{lem} \label{lem_interm_y_ln_y}
        For any $a \in (0,1]$ and $\delta \geq 0$ such that $a - \delta \geq -1$, there holds
        \begin{equation*}
            (a - \delta) \ln \abs{a - \delta} - a \ln a \leq \delta ( 1 - \ln a ).
        \end{equation*}
    \end{lem}
    
    The conclusion follows from applying this lemma with $a = \abs{z}$ and $\delta = \abs{\tilde{z} - z}$.
\end{proof}

\begin{proof}[Proof of Lemma \ref{lem_interm_y_ln_y}]
    Take $a$ and $\delta$ satisfying the assumptions of the Lemma.
    \begin{itemize}
        \item If $\delta < a$, then $0 \leq a - \delta \leq a$, so in particular $\ln \abs{a - \delta} \leq \ln a$, which yields
        \begin{equation*}
            (a - \delta) \ln \abs{a - \delta} - a \ln a \leq (a - \delta) \ln a - a \ln a = - \delta \ln a.
        \end{equation*}
        \item If $\delta \geq 2 a$, in the same way, we have $a - \delta \leq - a < 0$, in particular $\ln \abs{a - \delta} \geq \ln a$ which yields
        \begin{equation*}
            (a - \delta) \ln \abs{a - \delta} - a \ln a \leq (a - \delta) \ln a - a \ln a = - \delta \ln a.
        \end{equation*}
        \item Otherwise, if $a \leq \delta < 2a$, then $-1 < \frac{a - \delta}{a} = 1 - \frac{\delta}{a} \leq 0$ and we can compute
        \begin{equation*}
            (a - \delta) \ln \abs{a - \delta} - a \ln a = (a - \delta) \ln \abs{\frac{a - \delta}{a}} - \delta \ln a \leq a - \delta \ln a \leq \delta (1 - \ln a).
        \end{equation*}
        where we have used the fact that $y \ln \abs{y} \leq 1$ for all $y \in [-1,0]$. \qedhere
    \end{itemize}
\end{proof}

Substituting $\tilde{z}$ by $g (x)$ and $z$ by $g_k (x)$ (which does not vanish) for some $k$, we see that the $\ln{\abs{z}^2}$ in the right-hand side becomes a quadratic function in $x$, which is totally harmless compared with the decay at infinity of the Gaussons, provided that we use this inequality with $g_k$ only where $g_k$ is "predominant".
To be more precise, we will apply such an estimate on $I_k$ defined by
\begin{equation*}
    I_k = \{ x \in \mathbb{R}^d, \abs{x - x_k} \leq \abs{x - x_j} \quad \forall j \neq k \}.
\end{equation*}
It is easy to see that, since the $x_i$ are all different from each other, there holds $\mathcal{L}_d ( I_j \cap I_k ) = 0$ for all $j \neq k$
where $\mathcal{L}_d$ is the Lebesgue measure in $\mathbb{R}^d$.
Moreover, there also holds $\mathbb{R}^d = \bigcup\limits_{j} I_j$,
so that
\begin{equation*}
    \int_{\mathbb{R}^d} = \sum_j \int_{I_j}.
\end{equation*}
However, we also need $\abs{g (x)}$ and $\abs{g_k (x)}$ to be smaller than $1$ everywhere in order to apply the previous lemma, which means that we need the $\omega_j$s to be negatively large enough, but we show that we can stick to this case.

\begin{prop} \label{prop_L_infty_bound_multigauss}
    Set $\omega \coloneqq \max_j \omega_j$.
    For all $x \in \mathbb{R}^d$, there holds
    \begin{equation*}
        \sum_j \abs{g_j (x)} \leq N e^{\omega}.
    \end{equation*}
\end{prop}

\begin{proof}
    It easily follows from the fact that for any $j$ and any $x \in \mathbb{R}^d$, there holds $\abs{g_j (x)} \leq e^{\omega_j} \leq e^\omega$
\end{proof}

We define $\tilde{g}_k \coloneqq N^{-1} e^{- \omega} \, g_k$ and $\tilde{g} \coloneqq N^{-1} e^{- \omega} \, g$, so that
\begin{gather}
    \abs{\tilde{g} (x)} \leq \sum_k \abs{\tilde{g}_k (x)} \leq 1, \notag \\
    \abs{g (x) \ln \abs{g (x)}^2 - \sum_{j = 1}^N g_j (x) \ln \abs{g_j (x)}^2} = N e^{\omega} \abs{\tilde{g} (x) \ln \abs{\tilde{g} (x)}^2 - \sum_{j = 1}^N \tilde{g}_j (x) \ln \abs{\tilde{g}_j (x)}^2}. \label{lien_diff_x_ln_x_g_tilde_g}
\end{gather}
Then, we can use Lemma \ref{lem_lip_z_log_z}, which leads to:

\begin{prop} \label{prop_est_diff_g_ln_g_tilde}
    For all $k \in \{ 1, \dots, N \}$ and $x \in \mathbb{R}^d$, there holds
    \begin{equation*}
        \abs{\tilde{g} (x) \ln \abs{\tilde{g} (x)}^2 - \sum_{j = 1}^N \tilde{g}_j (x) \ln \abs{\tilde{g}_j (x)}^2} \leq 2 \sum_{j \neq k} \abs{\tilde{g}_j (x)} \Bigl[ \delta \omega_j + \delta \omega_k + 3 + 2 \ln{N} + \lambda_+ \abs{x - x_k}^2 + \lambda_+ \abs{x - x_j}^2 \Bigr], \label{est_diff_g_ln_g_tilde}
    \end{equation*}
    where $\delta \omega_j \coloneqq \omega - \omega_j$ for all $j$.
\end{prop}

\begin{proof}
    We recall that for all $j$,
    \begin{equation*}
        \abs{\tilde{g}_j (x)} = \exp \Bigl[ - \delta \omega_j - \ln{N} - (x - x_j)^\top \Lambda_j (x - x_j) \Bigr],
    \end{equation*}
    which also yields that
    \begin{equation*}
        - \ln{\abs{\tilde{g}_j (x)}} \leq \delta \omega_j + \ln{N} + \lambda_+ \abs{x - x_j}^2.
    \end{equation*}
    Then, we can easily compute
    \begin{align*}
        \abs{\tilde{g} (x) \ln \abs{\tilde{g} (x)} - \sum_{j = 1}^N \tilde{g}_j (x) \ln \abs{\tilde{g}_j (x)}} &\leq \abs*[\Big]{\tilde{g} (x) \ln \abs{\tilde{g} (x)} - \tilde{g}_k (x) \ln \abs{\tilde{g}_k (x)}} + \abs{\sum_{j \neq k} \tilde{g}_j (x) \ln \abs{\tilde{g}_j (x)}} \\
            &\leq \abs{\tilde{g} (x) - \tilde{g}_k (x)} \Bigl[ 3 - \ln \abs{\tilde{g}_k (x)} \Bigr] + \sum_{j \neq k} \abs{\tilde{g}_j (x)} \abs{\ln \abs{\tilde{g}_j (x)}} \\
            &\leq \abs{\sum_{j \neq k} \tilde{g}_j (x)} \Bigl[ 3 - \ln \abs{\tilde{g}_k (x)} \Bigr] + \sum_{j \neq k} \abs{\tilde{g}_j (x)} \abs{\ln \abs{\tilde{g}_j (x)}} \\
            &\leq \sum_{j \neq k} \abs{\tilde{g}_j (x)} \Bigl[ 3 - \ln \abs{\tilde{g}_k (x)} - \ln \abs{\tilde{g}_j (x)} \Bigr].
    \end{align*}
    The conclusion follows.
\end{proof}

Thanks to \eqref{lien_diff_x_ln_x_g_tilde_g}, we can readily come back in terms of $g$ and $g_j$.

\begin{cor} \label{cor_est_diff_g_ln_g}
    For all $k \in \{ 1, \dots, N \}$ and $x \in \mathbb{R}^d$, there holds
    \begin{equation}
        \abs{g (x) \ln \abs{g (x)}^2 - \sum_{j = 1}^N g_j (x) \ln \abs{g_j (x)}^2} \leq 2 \sum_{j \neq k} \abs{g_j (x)} \Bigl[ \delta \omega_j + \delta \omega_k + 3 + 2 \ln{N} + \lambda_+ \abs{x - x_k}^2 + \lambda_+ \abs{x - x_j}^2 \Bigr]. \label{est_diff_g_ln_g}
    \end{equation}
\end{cor}

Thus, we can estimate this difference in $L^2 (I_k)$ norm.

\begin{prop} \label{prop_L2_I_k_est}
    If $\varepsilon \leq \varepsilon_0$ (where $\varepsilon_0$ is defined in Lemma \ref{lem_diff_L_log_L_sum_gaussian}) then
    \begin{equation*}
        \norm{g (x) \ln \abs{g (x)}^2 - \sum_{j = 1}^N g_j (x) \ln \abs{g_j (x)}^2}_{L^2 (I_k)} \leq C_d N \, \frac{\lambda_+}{\varepsilon^{\frac{d}{2}+1} \sqrt{\lambda_-}} e^{\omega-\frac{\lambda_-}{8 \varepsilon^2}}.
    \end{equation*}
\end{prop}

To prove this result, we need to use Corollary \ref{cor_est_diff_g_ln_g} and estimate the $L^2 (I_k)$ norm of each term of the sum in \eqref{est_diff_g_ln_g}. For this, we will use the following lemma:

\begin{lem} \label{lem_int_error_gauss_2}
    For any $d \in \mathbb{N}^*$, there exists $C_d > 0$ such that for all $\gamma > 0$, $R > 0$ and $x_0 \in \mathbb{R}^d$ such that $R \geq \gamma^{-\frac{1}{2}}$ and $\abs{x_0} \leq 2R$, there holds
    \begin{align*}
        \int_{\mathbb{R}^d \setminus B(0, R)} \abs{x}^4 e^{-\gamma \abs{x}^2} \diff x &\leq C_d \frac{R^{d+2}}{\gamma} e^{-\gamma R^2},
        \\
        \int_{\mathbb{R}^d \setminus B(0, R)} e^{-\gamma \abs{x}^2} \diff x &\leq C_d \frac{R^{d-2}}{\gamma} e^{-\gamma R^2}, \\
        \int_{\mathbb{R}^d \setminus B(0, R)} \abs{x-x_0}^4 e^{-\gamma \abs{x}^2} \diff x &\leq C_d \frac{R^{d+2}}{\gamma} e^{-\gamma R^2}.
    \end{align*}
\end{lem}

Before proving this Lemma, we recall the usual and useful estimate for the Gauss error function.

\begin{lem} \label{lem_int_error_gauss}
    For any $y \geq 1$ and $\gamma > 0$, there holds
    \begin{gather*}
        \int_y^\infty e^{-\gamma x^2} \diff x < \frac{1}{2 \gamma y} e^{-\gamma y^2}.
    \end{gather*}
\end{lem}

\begin{proof}[Proof of Lemma \ref{lem_int_error_gauss}]
    We easily compute:
    \begin{equation*}
        \int_y^\infty e^{-\gamma x^2} \diff x < \int_y^\infty \frac{x}{y} e^{-\gamma x^2} \diff x = \frac{1}{y} \left[ - \frac{e^{-\gamma x^2}}{2 \gamma} \right]^\infty_y = \frac{1}{2 \gamma y} e^{-\gamma y^2}. \qedhere
    \end{equation*}
\end{proof}

\begin{proof}[Proof of Lemma \ref{lem_int_error_gauss_2}]
    For the first estimate, a radial change of variables yields
    \begin{equation*}
        \int_{\mathbb{R}^d \setminus B(0, R)} \abs{x}^4 e^{-\gamma \abs{x}^2} \diff x = C_d \int_R^\infty r^{3+d} e^{-\gamma r^2} \diff r = C_d \times J_{3+d},
    \end{equation*}
    where $J_m = \int_R^\infty r^m e^{-\gamma r^2} \diff r$ for any $m \in \mathbb{N}$. With an integration by parts, we get
    \begin{equation*}
        J_{m+2} = \frac{1}{2 \gamma} R^{m+1} e^{-\gamma R^2} + \frac{m+1}{2 \gamma} J_{m}.
    \end{equation*}
    Since we have $R > \gamma^{-\frac{1}{2}}$ and since there holds
    \begin{equation*}
        J_0 < \frac{1}{2 \gamma R} e^{-\gamma R^2}
        \qquad \textnormal{and} \qquad
        J_1 = \frac{1}{2 \gamma} e^{-\gamma R^2},
    \end{equation*}
    we can easily prove by induction that
    \begin{equation*}
        J_m \leq \frac{C_m}{\gamma} R^{m-1} e^{-\gamma R^2},
    \end{equation*}
    which leads to the first estimate.
    The second estimate can be proved in the same way.
    As for the third estimate, we use the fact that for $x \in \mathbb{R}^d$
    \begin{equation*}
        \abs{x-x_0}^4 \leq C_0 (\abs{x}^4 + \abs{x_0}^4),
    \end{equation*}
    which yields
    \begin{align*}
        \int_{\mathbb{R}^d \setminus B(0, R)} \abs{x-x_0}^4 e^{-\gamma \abs{x}^2} \diff x &\leq C_0 \biggl( \int_{\mathbb{R}^d \setminus B(0, R)} \abs{x}^4 e^{-\gamma \abs{x}^2} \diff x + \abs{x_0}^4 \int_{\mathbb{R}^d \setminus B(0, R)} e^{-\gamma \abs{x}^2} \diff x \biggr) \\
            &\leq C_d \biggl( \int_R^\infty r^{3+d} e^{-\gamma r^2} \diff r + \abs{x_0}^4 \int_R^\infty r^{d-1} e^{-\gamma r^2} \diff r \biggr) \\
            &\leq C_d \biggl( \frac{R^{d+2}}{\gamma} e^{-\gamma R^2} + \frac{\abs{x_0}^4 R^{d-2}}{\gamma} e^{-\gamma R^2} \biggr) \\
            &\leq C_d \frac{R^{d+2}}{\gamma} e^{-\gamma R^2}. \qedhere
    \end{align*}
\end{proof}

\begin{proof}[Proof of Proposition \ref{prop_L2_I_k_est}]
    Thanks to Corollary \ref{cor_est_diff_g_ln_g}, we have
    \begin{multline}
        \norm{g (x) \ln \abs{g (x)} - \sum_{j = 1}^N g_j (x) \ln \abs{g_j (x)}}_{L^2 (I_k)} \leq \sum_{j \neq k} \norm{g_j (x)}_{L^2 (I_k)} \Bigl[ \delta \omega_j + \delta \omega_k + 3 + 2 \ln{N} \Bigr] \\
        + \lambda_+ \norm{g_j (x) \abs{x - x_k}^2}_{L^2 (I_k)} + \lambda_+ \norm{g_j (x) \abs{x - x_j}^2}_{L^2 (I_k)}. \label{comp_4}
    \end{multline}
    For $j \neq k$, we know that for any $x \in I_k$, there holds
    \begin{equation*}
        \abs{x_j - x_k} \leq \abs{x - x_j} + \abs{x - x_k} \leq 2 \abs{x - x_j},
    \end{equation*}
    and thus
    \begin{equation*}
        \abs{x - x_j} \geq \frac{1}{2} \abs{x_j - x_k}.
    \end{equation*}
    Hence, $I_k \subset \mathbb{R}^d \setminus B (x_j, R_j^k)$ where $R_j^k \coloneqq \frac{1}{2} \abs{x_j - x_k}$. Therefore, using also the fact that $y^\top \Lambda_j y \geq \lambda_- \abs{y}^2$ for all $y \in \mathbb{R}^d$, we get
    \begin{align*}
        \norm{g_j (x)}_{L^2 (I_k)}^2 &\leq \int_{\mathbb{R}^d \setminus B (x_j, R_j^k)} \exp \Bigl[ 2 \omega_j - 2 \lambda_- \abs{x - x_j}^2 \Bigr], \\
        \norm{g_j (x) \abs{x - x_k}^2}_{L^2 (I_k)}^2 &\leq \int_{\mathbb{R}^d \setminus B (x_j, R_j^k)} \abs{x - x_k}^4 \exp \Bigl[ 2 \omega_j - 2 \lambda_- \abs{x - x_j}^2 \Bigr], \\
        \norm{g_j (x) \abs{x - x_j}^2}_{L^2 (I_k)}^2 &\leq \int_{\mathbb{R}^d \setminus B (x_j, R_j^k)} \abs{x - x_j}^4 \exp \Bigl[ 2 \omega_j - 2 \lambda_- \abs{x - x_j}^2 \Bigr].
    \end{align*}
    Since $\omega_j \leq \omega$, and with a change of variable, we get
    \begin{align*}
        \norm{g_j (x)}_{L^2 (I_k)}^2 &\leq \int_{\mathbb{R}^d \setminus B (0, R_j^k)} \exp \Bigl[ 2 \omega - 2 \lambda_- \abs{y}^2 \Bigr], \\
        \norm{g_j (x) \abs{x - x_k}^2}_{L^2 (I_k)}^2 &\leq \int_{\mathbb{R}^d \setminus B (0, R_j^k)} \abs{y - (x_k - x_j)}^4 \exp \Bigl[ 2 \omega - 2 \lambda_- \abs{y}^2 \Bigr], \\
        \norm{g_j (x) \abs{x - x_j}^2}_{L^2 (I_k)}^2 &\leq \int_{\mathbb{R}^d \setminus B (0, R_j^k)} \abs{y}^4 \exp \Bigl[ 2 \omega - 2 \lambda_- \abs{y}^2 \Bigr].
    \end{align*}
    Now, we apply Lemma \ref{lem_int_error_gauss_2} in order to estimate all the $L^2 (I_k)$ norms of the right-hand side. However, we need to check the assumptions. We already know that $\abs{x_k - x_j} = 2 R_j^k$. Moreover, there holds from the same equality and with the definition of $\varepsilon$
    \begin{equation}
        R_j^k \geq \frac{1}{2 \varepsilon}. \label{ineg_R_eps}
    \end{equation}
    The fact that $\varepsilon_0 \leq \sqrt{\frac{\lambda_-}{2}}$ yields $R_j^k \geq (2 \lambda_-)^{- \frac{1}{2}}$.
    Along with the fact that $\omega_j \leq \omega$, this leads to
    \begin{align*}
        \norm{g_j (x)}_{L^2 (I_k)}^2 &\leq C_d \frac{(R_j^k)^{d-2}}{\lambda_-} e^{2 \omega -2 \lambda_- (R_j^k)^2}, \\
        \norm{g_j (x) \abs{x - x_k}^2}_{L^2 (I_k)}^2 &\leq C_d \frac{(R_j^k)^{d+2}}{\lambda_-} e^{2 \omega -2 \lambda_- (R_j^k)^2}, \\
        \norm{g_j (x) \abs{x - x_j}^2}_{L^2 (I_k)}^2 &\leq C_d \frac{(R_j^k)^{d+2}}{\lambda_-} e^{2 \omega -2 \lambda_- (R_j^k)^2}.
    \end{align*}
    Since we also assumed
    \begin{equation*}
        \varepsilon_0 \leq \frac{\sqrt{\lambda_+}}{\max ( \sqrt{\delta \omega + 1}, \sqrt{\ln{N}} )},
    \end{equation*}
    \eqref{ineg_R_eps} leads to
    \begin{equation*}
        \max ( \delta \omega + 1, \ln{N} ) \leq \frac{\lambda_+}{\varepsilon^2} \leq 4 \lambda_+ (R_j^k)^2,
    \end{equation*}
    so that
    \begin{align*}
        \Bigl[ \delta \omega_j + \delta \omega_k + 3 + 2 \ln{N} \Bigr] \norm{g_j (x)}_{L^2 (I_k)} &\leq C_d \, \lambda_+ \frac{(R_j^k)^{\frac{d}{2}+1}}{\sqrt{\lambda_-}} e^{\omega - \lambda_- (R_j^k)^2}, \\
        \lambda_+ \norm{g_j (x) \abs{x - x_k}^2}_{L^2 (I_k)} &\leq C_d \, \lambda_+ \frac{(R_j^k)^{\frac{d}{2}+1}}{\sqrt{\lambda_-}} e^{\omega - \lambda_- (R_j^k)^2}, \\
        \lambda_+ \norm{g_j (x) \abs{x - x_j}^2}_{L^2 (I_k)} &\leq C_d \, \lambda_+ \frac{(R_j^k)^{\frac{d}{2}+1}}{\sqrt{\lambda_-}} e^{\omega - \lambda_- (R_j^k)^2}.
    \end{align*}
    Moreover, we know that the function $f_d (\xi) \coloneqq \xi^{\frac{d}{2} + 1} e^{- \lambda_- \xi^2}$ is decreasing for \\ $\xi \in \biggl[ \sqrt{\frac{\frac{d}{2} + 1}{2 \lambda_-}}, \infty \biggr)$. Therefore, as there also holds $\varepsilon_0 \leq \sqrt{\frac{\lambda_-}{d+2}}$, then \eqref{ineg_R_eps} also yields
    \begin{equation*}
        R_j^k \geq \frac{1}{2 \varepsilon} \geq \sqrt{\frac{\frac{d}{2} + 1}{2 \lambda_-}},
    \end{equation*}
    hence
    \begin{align*}
        \Bigl[ \delta \omega_j + \delta \omega_k + 3 + 2 \ln{N} \Bigr] \norm{g_j (x)}_{L^2 (I_k)} &\leq C_d \, \frac{\lambda_+}{\varepsilon^{\frac{d}{2}+1} \sqrt{\lambda_-}} e^{\omega -\frac{\lambda_-}{4 \varepsilon^2}}, \\
        \lambda_+ \norm{g_j (x) \abs{x - x_k}^2}_{L^2 (I_k)} &\leq C_d \, \frac{\lambda_+}{\varepsilon^{\frac{d}{2}+1} \sqrt{\lambda_-}} e^{\omega -\frac{\lambda_-}{4 \varepsilon^2}}, \\
        \lambda_+ \norm{g_j (x) \abs{x - x_j}^2}_{L^2 (I_k)} &\leq C_d \, \frac{\lambda_+}{\varepsilon^{\frac{d}{2}+1} \sqrt{\lambda_-}} e^{\omega -\frac{\lambda_-}{4 \varepsilon^2}}.
    \end{align*}
    We get the result by putting these into \eqref{comp_4}.
\end{proof}

\begin{proof}[Proof of Lemma \ref{lem_diff_L_log_L_sum_gaussian}]
    Define $\varepsilon_0$ as in Proposition \ref{prop_L2_I_k_est}.
    By definition of the sets $I_k$ ($k \in \{ 1, \dots, N \}$) and using Proposition \ref{prop_L2_I_k_est} for $\varepsilon \leq \varepsilon_0$, we get
    
    \begin{align*}
        \norm{g \ln \abs{g} - \sum_{j = 1}^N g_j \ln \abs{g_j}}_{L^2 (\mathbb{R})}^2 &= \sum_k \norm{g \ln \abs{g} - \sum_{j = 1}^N g_j \ln \abs{g_j}}_{L^2 (I_k)}^2 \\
        &\leq \sum_k C_d N^2 \frac{(\lambda_+)^2}{\varepsilon^{d+2} \lambda_-} e^{2 \omega -\frac{\lambda_-}{2 \varepsilon^2}} \\
        &\leq C_d N^3 \frac{(\lambda_+)^2}{\varepsilon^{d+2} \lambda_-} e^{2 \omega -\frac{\lambda_-}{2 \varepsilon^2}}. \qedhere
    \end{align*}
\end{proof}

\bibliographystyle{abbrv}
\bibliography{sample}

\end{document}